\documentclass[a4paper]{amsart}

\usepackage{amsmath}
\usepackage{amssymb}
\usepackage{bm} 
\usepackage{mathrsfs} 
\usepackage{tensor} 
\usepackage{braket} 
\usepackage{physics} 

\DeclareMathOperator{\vol}{vol}
\DeclareMathOperator{\Vol}{Vol}

\DeclareMathOperator{\Ric}{Ric}

\DeclareMathOperator{\Hom}{Hom}

\DeclareMathOperator{\Aff}{Aff}

\numberwithin{equation}{section}

\usepackage{hyperref}
\usepackage{amsthm} 

\usepackage{cleveref} 

\usepackage{autonum} 

\crefname{theorem}{Theorem}{Theorems}
\crefname{lemma}{Lemma}{Lemmas}
\crefname{proposition}{Proposition}{Propositions}
\crefname{corollary}{Corollary}{Corollaries}
\crefname{claim}{Claim}{Claims}
\crefname{assumption}{Assumption}{Assumptions}
\crefname{definition}{Definition}{Definitions}
\crefname{remark}{Remark}{Remarks}
\crefname{example}{Example}{Examples}
\crefname{equation}{}{}
\crefname{section}{Section}{Sections}
\crefname{subsection}{Section}{Sections}
\crefname{appendix}{Appendix}{Appendices}

\crefname{condition}{}{}

\creflabelformat{equation}{(#2#1#3)}
\creflabelformat{condition}{(#2#1#3)}

\theoremstyle{plain}
\newtheorem{theorem}{Theorem}[section]
\newtheorem{lemma}[theorem]{Lemma}
\newtheorem{proposition}[theorem]{Proposition}
\newtheorem{corollary}[theorem]{Corollary}

\theoremstyle{definition}
\newtheorem{definition}[theorem]{Definition}

\theoremstyle{remark}

\newtheorem{example}[theorem]{Example}

\usepackage[non-sorted-cites,initials]{amsrefs}



\AtBeginDocument{%
	\def\MR#1{}
}

\begin{document}
\title{Ambient constructions for Sasakian $\eta$-Einstein manifolds}
\author{Yuya Takeuchi}
\address{Graduate School of Mathematical Sciences \\ The University of Tokyo \\ 3-8-1 Komaba, Meguro, Tokyo 153-8914 Japan}
\email{ytake@ms.u-tokyo.ac.jp}

\begin{abstract}
The theory of ambient spaces is useful
to define CR invariant objects,
such as CR invariant powers of the sub-Laplacian,
the $P$-prime operator, and $Q$-prime curvature.
However in general,
it is difficult to write down these objects
in terms of the Tanaka-Webster connection.
In this paper,
we give those explicit formulas for CR manifolds satisfying an Einstein condition,
called Sasakian $\eta$-Einstein manifolds.
As an application,
we study properties of the first and the second variation
of the total $Q$-prime curvature at Sasakian $\eta$-Einstein manifolds.
\end{abstract}

\maketitle

\tableofcontents

\section{Introduction}
\label{section:introduction}

CR geometry has very similar properties to conformal geometry.
As an important example,
the theory of ambient spaces gives a powerful tool for both geometries.
An ambient space is an asymptotically Ricci-flat space
determined for a conformal or CR structure.
Many invariants for both geometries have been constructed
as obstructions to harmonic extensions,
for example, GJMS operators and the $Q$-curvature.

In conformal geometry,
a conformal class containing an Einstein metric plays an important role.
Several authors have studied relations between invariants constructed by using ambient spaces
and Einstein metrics,
while there exist few studies for corresponding relations in CR geometry.
The aim of this paper
is to reveal relations between the theory of ambient spaces
and CR manifolds satisfying an Einstein condition,
called Sasakian $\eta$-Einstein manifolds.

Via Fefferman construction,
we can see CR geometry as a special class of even dimensional conformal geometry;
we first recall the latter.
Let $N$ be a smooth manifold of even dimension $n$,
and $[g]$ a conformal class on $N$.
The metric bundle $\mathcal{G}$ is the principal $\mathbb{R}_{+}$-bundle over $N$
whose sections are the metrics in the conformal class.
The \emph{ambient space} is the space $\mathcal{G} \times \mathbb{R}$ 
with an asymptotically Ricci-flat metric $\widetilde{g}$ homogeneous of degree $2$
with respect to the $\mathbb{R}_{+}$-action.
We can define GJMS operators
as obstructions to harmonic extensions for the Laplacian of $\widetilde{g}$.
For $1 \leq k \leq n/2$,
the \emph{$k$-th GJMS operator} $P_{k}$
is a conformally invariant self-adjoint differential operator acting on densities.
The critical \emph{$Q$-curvature} is a smooth density 
determined for each representative of $[g]$.
It is not a conformal invariant,
but its integral, the \emph{total $Q$-curvature},
defines a global conformal invariant of $(N, [g])$.
See~\cite{Fefferman-Graham12} and references therein for detail.

Now assume that there exists an Einstein metric $g$ in $[g]$.
Gover~\cite{Gover06} and Fefferman-Graham~\cite{Fefferman-Graham12} proved that 
the $k$-th GJMS operator is decomposed into $k$ factors,
and each factor is the sum of the Laplacian and a constant depending on $n$
and the Einstein constant.
They also showed that
the critical $Q$-curvature with respect to $g$ is
a constant depending only on $n$ and the Einstein constant.
Moreover,
the variation of the total $Q$-curvature under deformations of conformal structures
is well-understood:
the first variation of the total $Q$-curvature at $[g]$ is always zero~\cite{Graham-Hirachi05},
and the second variation is written in terms of the Lichnerowicz Laplacian~\cite{Matsumoto13};
see also~\cite{Guillarmou-Moroianu-Schlenker}.

Let $M$ be a strictly pseudoconvex real hypersurface in an $(n+1)$-dimensional complex manifold $X$, 
which has the natural CR structure $T^{1, 0} M$ induced from the complex structure on $X$.
Consider the canonical bundle $\mathcal{X}$ of $X$ with the zero section removed,
which is a principal $\mathbb{C}^{\times}$-bundle over $X$.
Denote by $\mathcal{M}$ the restriction of $\mathcal{X}$ on $M$.
From the $\mathbb{C}^{\times}$-action,
we define the sheaf of homogeneous functions 
of degree $(w, w') \in \mathbb{R}^{2}$ on $\mathcal{X}$ and $\mathcal{M}$,
which is written as $\widetilde{\mathscr{E}}(w, w')$ and $\mathscr{E}(w, w')$ respectively.
A \emph{Fefferman defining function}
is a defining function $\bm{\rho} \in \widetilde{\mathscr{E}}(1, 1)$ of $\mathcal{M}$
satisfying the complex Monge-Amp\`{e}re equation 
\begin{equation}
	(\sqrt{-1} \partial \overline{\partial} \bm{\rho})^{n + 2}
	= - (\sqrt{-1})^{(n + 2)^{2}} \frac{(n + 1)!}{n + 2} (1+ \bm{\mathcal{O}} \bm{\rho}^{n + 2})
	d \bm{\zeta} \wedge \overline{d \bm{\zeta}},
\end{equation}
where $\bm{\mathcal{O}} \in \widetilde{\mathscr{E}}(- n - 2)$,
called the \emph{obstruction function},
and $\bm{\zeta}$ is the tautological $(n +1, 0)$-form on $\mathcal{X}$.
This defining function induces a Lorentz-K\"{a}hler metric
that is Ricci-flat to some order.
The CR counterparts of GJMS operators are \emph{CR invariant powers of the sub-Laplacian}.
For $(w, w') \in \mathbb{R}^{2}$ with $w + w' + n + 1 \in \{1, 2, \dots , n + 1 \}$,
Gover-Graham~\cite{Gover-Graham05}
defined a CR invariant operator $\bm{P}_{w, w'}$
whose leading part agrees with a power of the sub-Laplacian
acting on $\mathscr{E}(w, w')$ by using ambient spaces.
The natural CR analog of the critical $Q$-curvature in conformal geometry
is the \emph{$Q$-prime curvature}~\cite{Case-Yang13,Hirachi14}.
A contact form $\theta$ on $M$ is said to be \emph{pseudo-Einstein} 
if there exists a Hermitian metric $\bm{h}$ of $K_{X}$
such that $\bm{h}$ is flat on the pseudoconvex side
and
\begin{equation}
	\theta = (\sqrt{-1} / 2) (\overline{\partial} - \partial)(\bm{\rho} \cdot \bm{h}^{-1 / (n + 2)}) |_{M}
\end{equation}
holds;
here we consider $\bm{h}$ as a positive function in $\widetilde{\mathscr{E}}(n + 2, n + 2)$.
The $Q$-prime curvature $\bm{Q}'_{\theta}$ is a density defined for a pseudo-Einstein contact form $\theta$.
The transformation law of $\bm{Q}'_{\theta}$ under a conformal change
is given in terms of $\bm{P}_{0,0}$
and the \emph{$P$-prime operator} $\bm{P}'_{\theta}$.
Its integral, the \emph{total $Q$-prime curvature},
defines a global CR invariant of $M$~\cite{Marugame18},
which is denoted by $\overline{Q}'(M)$.

Sasakian $\eta$-Einstein manifolds are
CR manifolds satisfying an Einstein condition.
Let $(S, T^{1,0}S)$ be a $(2n + 1)$-dimensional strictly pseudoconvex CR manifold
and $\eta$ a contact form on $S$.
The triple $(S, T^{1,0} S, \eta)$ is called a \emph{Sasakian manifold}
if its Tanaka-Webster torsion vanishes identically.
Then the cone $C(S) = \mathbb{R}_{+} \times S$ of $S$ has a canonical complex structure,
and the level set $\{r =1\}$, where $r$ is the coordinate of $\mathbb{R}_{+}$,
is canonically isomorphic to $S$ as a CR manifold; 
in the following,
we identify $S$ with this level set.
A Sasakian manifold $(S, T^{1,0}S, \eta)$ is called a \emph{Sasakian $\eta$-Einstein manifold}
if its Tanaka-Webster Ricci curvature is a constant multiple of the Levi form;
in this case, we call this constant the \emph{Einstein constant} of $(S, T^{1,0}S, \eta)$.

Now we give the statements of our results, 
which reveal connections between the theory of ambient spaces in CR geometry
and Sasakian $\eta$-Einstein manifolds.

First we give formulas of CR invariant powers of the sub-Laplacian
and the $P$-prime operator on Sasakian $\eta$-Einstein manifolds.
To simplify the formulas,
we use unbold differential operators $P_{w, w'}$ and $P'_{\eta}$ acting on functions on $S$
instead of $\bm{P}_{w, w'}$ and $\bm{P}'_{\eta}$
that act on densities (\cref{def:operators-on-S}).
The operators $P_{w, w'}$ and $P'_{\eta}$ have expressions
in terms of the sub-Laplacian $\Delta_{b}$ and the Reeb vector field $\xi$.

\begin{theorem} \label{thm:formula-for-CR-invariant-powers-of-sub-Laplacian}
	Let $(S, T^{1,0}S, \eta)$ be a Sasakian $\eta$-Einstein manifold of dimension $2n + 1$
	with Einstein constant $(n+1) \lambda$.
	Then $P_{w, w'}$ for $k = w + w' + n + 1$ has the formula
	\begin{equation}
		P_{w,w'} = \prod_{j=0}^{k-1} L_{w'-w + k- 2j -1}.
	\end{equation}
	Here $L_{\mu}$ is the differential operator on $S$ defined by
	\begin{equation}
		L_{\mu} = \frac{1}{2} \Delta_{b} + \frac{\sqrt{-1}}{2}\mu \xi
		+ \frac{1}{4} \lambda (n - \mu)(n + \mu).
	\end{equation}
\end{theorem}

\begin{theorem} \label{thm:formula-for-P-prime-operator}
	Let $(S, T^{1,0}S, \eta)$ be as in \cref{thm:formula-for-CR-invariant-powers-of-sub-Laplacian}.
	For any CR pluriharmonic function $\Upsilon$ on $S$,
	\begin{equation}
		P'_{\eta} \Upsilon = n^{-1} P_{-1,-1} (\Delta_{b}^{2} + n^{2} \lambda \Delta_{b}) \Upsilon.
	\end{equation}
\end{theorem}

In the case of the sphere or the Heisenberg group, 
\cref{thm:formula-for-CR-invariant-powers-of-sub-Laplacian,thm:formula-for-P-prime-operator}
were already obtained by Graham~\cite{Graham84}
and Branson-Fontana-Morpurgo~\cite{Branson-Fontana-Morpurgo13},
respectively.
Remark that to compare our result to that in~\cite{Branson-Fontana-Morpurgo13}, 
we need to use the fact that any CR pluriharmonic function
is annihilated by $\Delta_{b}^{2} + n^{2} \xi^{2}$ on Sasakian manifolds.

We also compute the $Q$-prime curvature of Sasakian $\eta$-Einstein manifolds.
Similar to the above,
we use unbold $Q'_{\eta}$ instead of $\bm{Q'}_{\eta}$;
see \cref{eq:unbold-Q-prime-curvature}.

\begin{theorem} \label{thm:formula-for-Q-prime-curvature}
	Let $(S, T^{1,0}S, \eta)$ be as in \cref{thm:formula-for-CR-invariant-powers-of-sub-Laplacian}.
	Then the $Q$-prime curvature $Q'_{\eta}$ is written as
	\begin{equation}
		Q'_{\eta} = 2 (n!)^{2} \lambda^{n + 1}.
	\end{equation}
\end{theorem}

As an application of this formula,
we will compute the total $Q$-prime curvature for some Sasakian $\eta$-Einstein manifolds
in \cref{section:computation-of-total-Q-prime-curvature}.

Note that 
Case-Gover~\cite{Case-Gover17} also obtained the same results
as in \cref{thm:formula-for-CR-invariant-powers-of-sub-Laplacian,thm:formula-for-P-prime-operator,thm:formula-for-Q-prime-curvature}
in terms of tractors.

Finally,
we consider the variation of the total $Q$-prime curvature
under deformations of real hypersurfaces.

\begin{proposition} \label{prop:first-variation-for-SeE}
	Let $(S, T^{1,0}S, \eta)$ be a closed $(2n + 1)$-dimensional Sasakian $\eta$-Einstein manifold, 
	and $(M_{t})_{t \in (-1, 1)}$ be a smooth family of closed real hypersurfaces in $C(S)$
	such that $M_{0} = S$.
	Then the first variation of the total $Q$-prime curvature vanishes; 
	that is,
	\begin{equation}
		\left. \frac{d}{dt} \right|_{t=0} \overline{Q}'(M_{t}) = 0.
	\end{equation}
\end{proposition}

This proposition follows from the fact that 
the obstruction function $\bm{\mathcal{O}}$ vanishes for $S$;
see \cref{prop:Fefferman-defining-function-of-SeE}.
Moreover,
the second variation is written in terms of $P_{1,1}$,
proved by Hirachi-Marugame-Matsumoto~\cite[Theorem 1.2]{Hirachi-Marugame-Matsumoto17}.
Spectral properties of $\Delta_{b}$ and $\xi$ give the following

\begin{theorem} \label{thm:second-variation-for-non-negative-SeE}
	Let $(S, T^{1,0}S, \eta)$ and $(M_{t})_{t \in (-1, 1)}$
	be as in \cref{prop:first-variation-for-SeE}.
	Assume that $n = 1$ or the Einstein constant is non-negative.
	Then, the second variation of the total $Q$-prime curvature is non-positive;
	that is,
	\begin{equation}
		\left. \frac{d^{2}}{dt^{2}} \right|_{t=0} \overline{Q}'(M_{t}) \leq 0.
	\end{equation}
	Moreover, the equality holds if and only if
	$(M_{t})_{ \in (-1, 1)}$ is infinitesimally trivial as a deformation of CR structures
	(see \cref{def:infinitesimally-trivial}).
\end{theorem}

On the other hand,
the conclusion of \cref{thm:second-variation-for-non-negative-SeE}
does not hold for Sasakian $\eta$-Einstein manifolds
of dimension greater than three and with negative Einstein constant.

\begin{theorem} \label{thm:second-variation-for-negative-SeE}
	For each integer $n \geq 2$,
	there exist a closed Sasakian $\eta$-Einstein manifold of dimension $2n + 1$
	with negative Einstein constant
	and an infinitesimally non-trivial smooth deformation
	such that the second variation of the total $Q$-prime curvature along this deformation is equal to zero.
	If $n$ is even,
	one can also find an example of a Sasakian $\eta$-Einstein manifold and a smooth deformation
	such that the second variation of the total $Q$-prime curvature along this deformation is positive.
\end{theorem}

This paper is organized as follows.
In \cref{section:preliminaries},
we recall basic concepts of CR manifolds, ambient spaces, and Sasakian manifolds.
\cref{section:construction-of-Fefferman-defining-function} is devoted to
the construction of a Fefferman defining function.
In \cref{section:proof-of-factorization-theorem},
we provide the proofs of \cref{thm:formula-for-CR-invariant-powers-of-sub-Laplacian,thm:formula-for-P-prime-operator}.
\cref{section:variation-of-total-Q-prime-curvature} deals with
the variation of the total $Q$-prime curvature.
In \cref{section:computation-of-total-Q-prime-curvature},
we compute the total $Q$-prime curvature for some examples.

\medskip

\noindent
{\em Notation.}
We use Einstein's summation convention and assume that 
\begin{itemize}
	\item uppercase Latin indices $A, B, C, \dots$ run from $0$ to $n + 1$;
	\item lowercase Latin indices $a, b, c, \dots$ run from $1$ to $n + 1$; 
	\item lowercase Greek indices $\alpha, \beta, \gamma, \dots$ run from $1$ to $n$. 
\end{itemize}
We use $\theta$ and $T$ (resp. $\eta$ and $\xi$) 
for a contact form and the Reeb vector field
on a CR manifold (resp. Sasakian manifold).

\medskip

\section{Preliminaries}
\label{section:preliminaries}

\subsection{CR manifolds} \label{subsection:CR-manifold}

Let $M$ be a smooth $(2n + 1)$-dimensional manifold without boundary.
A \emph{CR structure} is a complex $n$-dimensional subbundle $T^{1,0}M$
of the complexified tangent bundle $TM \otimes \mathbb{C}$ such that
\begin{itemize}
	\item $T^{1,0}M \cap T^{0,1}M = 0$, where $T^{0,1}M = \overline{T^{1,0}M}$;
	\item $[\Gamma(T^{1,0}M), \Gamma(T^{1,0}M)] \subset \Gamma(T^{1,0}M)$.
\end{itemize}
For example, 
if $M$ is a real hypersurface in a complex manifold $X$,
then $M$ has the natural CR structure
\begin{equation}
	T^{1,0} M = (TM \otimes \mathbb{C}) \cap T^{1,0}X.
\end{equation}
Set $H M = \Re T^{1,0}M$
and let $J \colon H M \to H M$ be the unique complex structure on $H M$ 
such that
\begin{equation}
	T^{1,0}M = \ker(J-\sqrt{-1} \colon H M \otimes \mathbb{C}
	\to H M \otimes \mathbb{C}).
\end{equation}
In the following, assume that there exists a nowhere vanishing $1$-form $\theta$
that annihilates $H M$.
The \emph{Levi form} $\mathcal{L}_{\theta}$ with respect to $\theta$ is the Hermitian form
on $T^{1,0} M$ defined by
\begin{equation}
	\mathcal{L}_{\theta}(Z, W) = - \sqrt{-1} \, d\theta(Z, \overline{W}), \quad Z,W \in T^{1,0} M.
\end{equation}
We consider only \emph{strictly pseudoconvex CR manifolds},
i.e., CR manifolds that has a positive definite Levi form
for some $\theta$;
such a $\theta$ is called a \emph{contact form}.
Denote by $T$ the \emph{Reeb vector field} with respect to $\theta$; 
that is, the unique vector field satisfying
\begin{equation}
	\theta(T) = 1, \quad \iota_{T} d\theta = 0.
\end{equation}
Then the tangent bundle has the decomposition
$TM = H M \oplus \mathbb{R} T$.
One can define a Riemannian metric $g_{\theta}$ on $M$ by 
\begin{equation}
	g_{\theta}(X,Y) = \frac{1}{2} d\theta(X, JY) + \theta(X) \theta(Y),
	\quad X,Y \in TM.
\end{equation}
Here, we extend $J$ to an endomorphism on $TM$ by $JT = 0$.
Take a local frame $(Z_{\alpha})$ of $T^{1, 0} M$.
Then we have a local frame $(T, Z_{\alpha}, Z_{\overline{\beta}} = \overline{Z_{\beta}})$
of $TM \otimes \mathbb{C}$,
and take the dual frame $(\theta, \theta^{\alpha}, \theta^{\overline{\beta}})$.
For this frame,
the $2$-form $d \theta$ is written as follows:
\begin{equation}
	d \theta = \sqrt{-1} \tensor{l}{_{\alpha} _{\overline{\beta}}}
	\theta^{\alpha} \wedge \theta^{\overline{\beta}},
\end{equation}
where $(\tensor{l}{_{\alpha} _{\overline{\beta}}})$ is a positive Hermitian matrix.
In the following,
we will use $\tensor{l}{_{\alpha} _{\overline{\beta}}}$
and its inverse $\tensor{l}{^{\alpha} ^{\overline{\beta}}}$
to lower and raise indices of various tensors.

A contact form $\theta$ induces a canonical connection $\nabla$,
called the \emph{Tanaka-Webster connection} with respect to $\theta$.
This connection is given by
\begin{equation}
	\nabla Z_{\alpha}
	= \tensor{\omega}{_{\alpha} ^{\beta}} Z_{\beta}, \quad
	\nabla Z_{\overline{\alpha}}
	= \tensor{\omega}{_{\overline{\alpha}} ^{\overline{\beta}}} Z_{\overline{\beta}}, \quad
	\nabla T = 0,
\end{equation}
and $\tensor{\omega}{_{\alpha} ^{\beta}},
\tensor{\omega}{_{\overline{\alpha}} ^{\overline{\beta}}}
= \overline{\tensor{\omega}{_{\alpha} ^{\beta}}}$
satisfy the following structure equations:
\begin{align}
	d \tensor{l}{_{\alpha} _{\overline{\beta}}}
	&= \tensor{\omega}{_{\alpha} _{\overline{\beta}}}
	+ \tensor{\omega}{_{\overline{\beta}} _{\alpha}}, \\
	d \theta^{\alpha}
	&= \theta^{\beta} \wedge \tensor{\omega}{_{\beta}^{\alpha}}
	+ \tensor{A}{^{\alpha} _{\overline{\beta}}} \, \theta \wedge \theta^{\overline{\beta}}.
\end{align}
Here the tensor
$\tensor{A}{_{\alpha} _{\beta}} = \overline{\tensor{A}{_{\overline{\alpha}} _{\overline{\beta}}}}$,
is symmetric and called the \emph{Tanaka-Webster torsion}.
We denote the components of successive covariant derivatives of a tensor
by subscripts preceded by the comma,
for example, $\tensor{K}{_{\alpha} _{\overline{\beta}} _{,} _{\gamma}}$;
we omit the comma if the derivatives are applied to functions.
With this notation,
introduce the operator $\overline{\partial}_{b}$ acting on $C^{\infty}(M)$ by
\begin{equation}
	\overline{\partial}_{b} f = \tensor{f}{_{\overline{\alpha}}} \theta^{\overline{\alpha}}.
\end{equation}
A smooth function $f$ is called a \emph{CR holomorphic function}
if $\overline{\partial}_{b} f = 0$.
A \emph{CR pluriharmonic function} is a real-valued smooth function
that is locally the real part of a CR holomorphic function.
We denote by $\mathscr{P}$ the space of CR pluriharmonic functions.
If $M$ is a real hypersurface in a complex manifold $X$,
then any CR holomorphic function (resp.\ CR pluriharmonic function)
can be extended to a holomorphic function (resp.\ pluriharmonic function)
on the pseudoconvex side.

The curvature form $\tensor{\Omega}{_{\alpha}^{\beta}}$ of $\nabla$
is given by
\begin{equation} \label{eq:curvature-form}
	\begin{split}
		\tensor{\Omega}{_{\alpha}^{\beta}}
		&= 
		\tensor{R}{_{\alpha} ^{\beta} _{\gamma} _{\overline{\delta}}}
		\theta^{\gamma} \wedge \theta^{\overline{\delta}}
		+ \tensor{A}{_{\alpha} _{\gamma} _{,} ^{\beta}} \theta^{\gamma} \wedge \theta
		- \tensor{A}{^{\beta} _{\overline{\gamma}} _{,} _{\alpha}}
		\theta^{\overline{\gamma}} \wedge \theta \\
		& \quad - \sqrt{-1} \tensor{A}{_{\alpha} _{\gamma}} \theta^{\gamma} \wedge \theta^{\beta}
		+ \sqrt{-1} \tensor{l}{_{\alpha} _{\overline{\gamma}}} \tensor{A}{^{\beta} _{\overline{\delta}}}
		\theta^{\overline{\gamma}} \wedge \theta^{\overline{\delta}}.
	\end{split}
\end{equation}
We call the tensor $\tensor{R}{_{\alpha} ^{\beta} _{\gamma} _{\overline{\delta}}}$
the \emph{Tanaka-Webster curvature}.
This tensor has the symmetry 
\begin{equation}
	\tensor{R}{_{\alpha} _{\overline{\beta}} _{\gamma} _{\overline{\delta}}}
	= \tensor{R}{_{\gamma} _{\overline{\beta}} _{\alpha} _{\overline{\delta}}}
	= \tensor{R}{_{\alpha} _{\overline{\delta}} _{\gamma} _{\overline{\beta}}}.
\end{equation}
Define the \emph{Tanaka-Webster Ricci curvature} $\tensor{\Ric}{_{\alpha}_{\overline{\beta}}}$ by
\begin{equation}
	\tensor{\Ric}{_{\alpha} _{\overline{\beta}}}
	= \tensor{R}{_{\alpha} ^{\gamma} _{\gamma} _{\overline{\beta}}}
	= \tensor{R}{_{\gamma} ^{\gamma} _{\alpha} _{\overline{\beta}}}
	= \tensor{R}{_{\alpha} _{\overline{\beta}} _{\gamma} ^{\gamma}}
	= \tensor{R}{_{\gamma} _{\overline{\beta}} _{\alpha} ^{\gamma}}.
\end{equation}

Commutators of derivatives are important for computation.
However, we postpone formulas of commutators until we introduce Sasakian manifolds,
since commutation relations are simplified for such CR manifolds.

\subsection{Ambient space}
\label{subsection:ambient-space}

Let $X$ be an $(n + 1)$-dimensional complex manifold
and $\pi_{\mathcal{X}} \colon \mathcal{X} = K_{X}^{\times} \to X$
the total space of the canonical bundle of $X$
with the zero section removed.
For $\mu \in \mathbb{C}^{\times}$,
define the dilation $\bm{\delta}_{\mu} \colon \mathcal{X} \to \mathcal{X}$
by the scalar multiplication $\bm{\delta}_{\mu}(\xi) = \mu^{n + 2} \xi$
for $\xi \in \mathcal{X}$.
Denote by $Z_{0}$ the holomorphic vector field generating $\bm{\delta}_{\mu}$;
that is, $Z_{0} = (d/d\mu)|_{\mu = 1} \bm{\delta}_{\mu}^{*}$.
A smooth function $\bm{f}$ on an open set of $\mathcal{X}$ is said to be
\emph{homogeneous of degree} $(w, w')$ for $w, w' \in \mathbb{R}$
if $Z_{0} \bm{f} = w \bm{f}$ and $\overline{Z_{0}} \bm{f} = w' \bm{f}$ hold.
We write $\widetilde{\mathscr{E}}(w, w')$
for the sheaf of smooth homogeneous functions of degree $(w, w')$. 
To simplify notation, write $\widetilde{\mathscr{E}}(w) = \widetilde{\mathscr{E}}(w, w)$.
Note that a Hermitian metric of $K_{X}$ can be identified with
a positive function in $\widetilde{\mathscr{E}}(n + 2)$.

Let $z = (z^{1}, \dots , z^{n + 1})$ be a local coordinate of $X$.
Then the fiber coordinate $\zeta$ of $\mathcal{X}$ is defined by
$\zeta dz^{1} \wedge \dots \wedge dz^{n + 1}$.
Choose a branch $z^{0} = \zeta^{1/(n + 2)}$,
which is called a \emph{branched fiber coordinate} in this paper.
Then $Z_{0}$ is equal to $z^{0} \partial_{z^{0}}$,
and for each homogeneous function $\bm{f}$ of degree $(w, w')$,
there exists a smooth function $f$ locally on $X$ such that
$\bm{f} = (z^{0})^{w} (\bar{z}^{0})^{w'} f$.

Let $M$ be a strictly pseudoconvex real hypersurface in $X$
and $\mathcal{M} = \pi_{\mathcal{X}}^{-1} (M)$.
Then $\mathcal{M}$ is a pseudoconvex real hypersurface in $\mathcal{X}$.
Denote by $\mathscr{E}(w, w')$
the sheaf of homogeneous functions of degree $(w, w')$ on $\mathcal{M}$.

For a defining function
$\bm{\rho} \in \widetilde{\mathscr{E}}(1)$ of $\mathcal{M}$,
the $(1,1)$-form $dd^{c} \bm{\rho}$ defines
a Lorentz-K\"{a}hler metric $\bm{g}[\bm{\rho}]$ in a neighborhood of $\mathcal{M}$,
where $d^{c} = (\sqrt{-1}/2)(\overline{\partial} - \partial)$.
We normalize $\bm{\rho}$ by a complex Monge-Amp\`ere equation.
Take the tautological $(n + 1, 0)$-form $\bm{\zeta}$ on $K_{X}$.
Then
\begin{equation} \label{eq:canonical-volume-form}
	\vol_{\mathcal{X}} =(\sqrt{-1})^{(n + 2)^{2}} d \bm{\zeta}
	\wedge \overline{d \bm{\zeta}}
\end{equation}
gives a volume form on $\mathcal{X}$.

\begin{proposition}[{\cite[Proposition 2.2]{Hirachi-Marugame-Matsumoto17}}]
\label{prop:Fefferman-defining-function}
	There exists a defining function $\bm{\rho} \in \widetilde{\mathscr{E}}(1)$
	of $\mathcal{M}$ such that
	\begin{equation} \label{eq:Fefferman-defining-function}
		(d d^{c} \bm{\rho})^{n + 2} = k_{n} (1+ \bm{\mathcal{O}} \bm{\rho}^{n + 2}) \vol_{\mathcal{X}},
	\end{equation}
	where $\bm{\mathcal{O}} \in \widetilde{\mathscr{E}}(-n - 2)$
	and $k_{n} = - (n + 1)!/(n + 2)$.
	Moreover, such a $\bm{\rho}$ is unique modulo $O(\bm{\rho}^{n + 3})$,
	and $\bm{\mathcal{O}}$ modulo $O(\bm{\rho})$ is independent of the choice of $\bm{\rho}$.
\end{proposition}

We call such a $\bm{\rho}$ a \emph{Fefferman defining function}
and the Lorentz-K\"{a}hler metric $\bm{g}[\bm{\rho}]$ with respect to $\bm{\rho}$
an \emph{ambient metric}.
The function $\bm{\mathcal{O}}$ is called the \emph{obstruction function}.

Next we introduce the pseudo-Einstein condition for a contact form on $M$,
which is necessary for the definition of the $Q$-prime curvature.
To this end,
recall a correspondence between Hermitian metrics of $K_{X}$ and defining functions of $M$. 
For a Hermitian metric $\bm{h}$ of $K_{X}$,
the function $\bm{\rho} \cdot \bm{h}^{- 1 / (n + 2)} \in \widetilde{\mathscr{E}}(0)$
gives a defining function of $M$.
Conversely,
let $\rho$ be a defining function of $M$.
Then $\bm{h}_{\rho} = (\bm{\rho} / \rho)^{n + 2} \in \widetilde{\mathscr{E}}(n + 2)$
defines a Hermitian metric of $K_{X}$ near $M$.
Moreover,
if $\rho$ is normalized by a contact form $\theta$
(i.e., $\theta = d^{c} \rho|_{M}$), 
then $\bm{h}_{\rho}|_{\mathcal{M}} \in \mathscr{E}(n + 2)$
depends only on $\theta$,
denoted by $\bm{h}_{\theta}$.
In particular,
if we fix a contact form $\theta$,
then multiplication by $\bm{h}_{\theta}^{- w/(n + 2)}$
defines an identification between $\mathscr{E}(w)$ and $C^{\infty}(M)$.

\begin{definition} \label{def:pseudo-Einstein}
	A contact form $\theta$ on $M$ is said to be \emph{pseudo-Einstein}
	if there exists a defining function $\rho$ normalized by $\theta$
	such that $\bm{h}_{\rho}$ is flat on the pseudoconvex side;
	see~\cite[Proposition 2.6]{Hirachi-Marugame-Matsumoto17} for equivalent conditions.
	Another contact form $\widehat{\theta} = e^{\Upsilon} \theta$ 
	is pseudo-Einstein if and only if $\Upsilon \in \mathscr{P}$.
\end{definition}

Before the end of this subsection,
note that there exists a canonical bijection
between $\mathscr{E}(- n - 1)$ and the space of volume forms on $M$.
For $\bm{\varphi} \in \mathscr{E}(- n - 1)$,
the $(2n + 1)$-form $\bm{\varphi} \, d^{c} \bm{\rho} \wedge (d d^{c} \bm{\rho})^{n}$
descends to a volume form on $M$.
Hence we denote by $\int_{M} \bm{\varphi}$
the integral of the volume form corresponding to $\bm{\varphi} \in \mathscr{E}(- n -1)$
for a closed CR manifold $M$.

\subsection{Deformation of CR structures}
\label{subsection:deformation-of-CR-structures}

This subsection deals with deformations of real hypersurfaces in a complex manifold
and corresponding deformations of CR structures;
we follow the argument in~\cite[Section 4]{Hirachi-Marugame-Matsumoto17}.

Let $M$ be a closed strictly pseudoconvex real hypersurface
in a complex manifold $X$ of dimension $n + 1$,
and $(M_{t})_{t \in (-1, 1)}$ be a smooth family of closed real hypersurfaces in $X$
such that $M_{0} = M$.
Take a Fefferman defining function $\bm{\rho}_{t}$ of $\mathcal{M}_{t} = \pi_{\mathcal{X}}^{-1} (M_{t})$
that is smooth in $t$.
Then $(d/dt)|_{t=0} \bm{\rho}_{t}|_{\mathcal{M}} \in \mathscr{E}(1)$
is independent of the choice of $\bm{\rho}_{t}$.
Conversely,
for any real-valued function $\bm{\varphi} \in \mathscr{E}(1)$,
there exists a smooth family $(M_{t})_{t \in (-1, 1)}$
such that $\bm{\varphi} = (d/dt)|_{t=0} \bm{\rho}_{t}|_{\mathcal{M}}$.
Thus the space of infinitesimal deformations of real hypersurfaces
is naturally parametrized by $\mathscr{E}(1)_{\mathbb{R}}$,
the space of real-valued functions in $\mathscr{E}(1)$.
On the other hand,
the \emph{space of infinitesimal deformations of CR structures},
denoted by $\mathscr{D}(M, T^{1,0} M)$,
is a linear subspace of $\Gamma(\Hom(T^{0,1} M, T^{1,0} M))$.
Each infinitesimal deformation of real hypersurfaces
induces that of CR structures.
This correspondence is represented by a differential operator
\begin{equation}
	D \colon \mathscr{E}(1) \to \mathscr{D}(M, T^{1,0} M),
\end{equation}
first introduced by Buchweitz-Millson~\cite[Chapter 8]{Buchweitz-Millson97}; 
see also~\cite[Section 4.]{Akahori-Garfield-Lee02}.
This operator also appears in a subcomplex
of the BGG sequence~\cite{Cap-Slovak-Soucek01}.
If $\bm{F} \in \mathscr{E}(1)$ is pure imaginary,
$D \bm{F}$ corresponds to
an infinitesimal deformation induced from an infinitesimal contact diffeomorphism on $M$,
which is studied by Cheng-Lee~\cite[Lemma 3.4]{Cheng-Lee90}.
From this observation,
we define a ``trivial deformation of real hypersurfaces'' from CR point of view.

\begin{definition} \label{def:infinitesimally-trivial}
	A smooth family $(M_{t})_{t \in (-1, 1)}$ of closed real hypersurfaces
	is said to be \emph{infinitesimally trivial as a deformation of CR structures}
	if
	\begin{equation}
		\left. \frac{d}{dt} \right|_{t=0} \bm{\rho}_{t}|_{\mathcal{M}} \in \Re \ker D.
	\end{equation}
\end{definition}

As stated in \cref{subsection:ambient-space},
a contact form $\theta$ on $M$ gives an identification between $\mathscr{E}(1)$ and $C^{\infty}(M)$.
Thus we obtain a differential operator
\begin{equation}
	D_{\theta} \colon C^{\infty}(M) \to \mathscr{D}(M, T^{1,0} M),
\end{equation}
written in terms of the Tanaka-Webster connection as follows:
\begin{equation}
	2 \tensor{(D_{\theta} F)}{_{\overline{\alpha}} ^{\beta}} 
	= \tensor{F}{_{\overline{\alpha}} ^{\beta}} - \sqrt{-1} \tensor{A}{_{\overline{\alpha}} ^{\beta}} F.
\end{equation}

\subsection{Ambient construction}
\label{subsection:ambient-construction}

In this subsection, we recall CR invariant powers of the sub-Laplacian,
the $P$-prime operator, and $Q$-prime curvature,
which are main subjects in this paper.

Let $\bm{\Delta}$ be the $\overline{\partial}$-Laplacian
with respect to an ambient metric $\bm{g}[\bm{\rho}]$.
This operator maps $\widetilde{\mathscr{E}}(w, w')$
to $\widetilde{\mathscr{E}}(w - 1, w' - 1)$.

\begin{lemma}[{\cite[Theorem 1.1]{Gover-Graham05}}]
	Let $(w, w') \in \mathbb{R}^{2}$ such that $k = w + w' + n + 1$ is a positive integer.
	Then for $\widetilde{\bm{f}} \in \widetilde{\mathscr{E}}(w, w')$,
	\begin{equation}
		(\bm{\Delta}^{k} \widetilde{\bm{f}}) |_{\mathcal{M}} \in \mathscr{E}(w - k, w' - k)
	\end{equation}
	depends only on $\bm{f} = \widetilde{\bm{f}} |_{\mathcal{M}}$
	and defines a differential operator
	\begin{equation}
		\bm{P}_{w, w'} \colon \mathscr{E}(w, w') \to \mathscr{E}(w - k, w' - k).
	\end{equation}
	Moreover, the operator $\bm{P}_{w, w'}$ is independent of
	the choice of a Fefferman defining function if $k \le n + 1$.
\end{lemma}

When $w = w' =1$, the operator $\bm{P}_{1,1}$ depends on the choice of a Fefferman defining function.
However, a slight modification gives a new CR invariant differential operator,
which is closely related to the variation of the total $Q$-prime curvature.

\begin{lemma}[{\cite[Lemma A.2]{Hirachi-Marugame-Matsumoto17}}]
\label{lem:super-critical-GJMS-operator}
	Let $\bm{\nabla}$ be the Levi-Civita connection
	with respect to the Lorentz-K\"{a}hler metric $\bm{g}[\bm{\rho}]$.
	Then for $\widetilde{\bm{f}} \in \widetilde{\mathscr{E}}(1)$, 
	\begin{equation}
		[\Re (\bm{\Delta}^{n + 1}
		\tensor{\bm{\nabla}}{^{A}} \tensor{\bm{\nabla}}{^{B}} \tensor{\bm{\nabla}}{_{A}} \tensor{\bm{\nabla}}{_{B}})
		\widetilde{\bm{f}}] |_{\mathcal{M}}
		\in \mathscr{E}(- n - 2)
	\end{equation}
	depends only on $\bm{f} = \widetilde{\bm{f}}|_{\mathcal{M}}$
	and defines a differential operator
	\begin{equation}
		\bm{R} \colon \mathscr{E}(1) \to \mathscr{E}(- n - 2).
	\end{equation}
	Moreover, $\bm{R}$ is independent of the choice of a Fefferman defining function.
	The operator $\bm{R}$ coincides with $\bm{P}_{1,1}$  
	if the obstruction function $\bm{\mathcal{O}}$ vanishes along $\mathcal{M}$.
\end{lemma}

To define the $Q$-prime curvature,
we need to fix a pseudo-Einstein contact form on $M$.

\begin{definition}[{\cite[Definition 5.4]{Hirachi14}}] \label{def:Q-prime-curvature}
	Let $\theta$ be a pseudo-Einstein contact form on $M$
	and $\rho$ a defining function of $M$ normalized by $\theta$
	such that $\bm{h}_{\rho}$ is flat on the pseudoconvex side.
	The \emph{$Q$-prime curvature} $\bm{Q}'_{\theta}$ is defined by
	\begin{equation}
		\bm{Q}'_{\theta}
		= (n + 2)^{-2} [\bm{\Delta}^{n + 1} (\log \bm{h}_{\rho})^{2}] |_{\mathcal{M}} \in \mathscr{E}(- n -1).
	\end{equation}
\end{definition}

If we take another pseudo-Einstein contact form
$\widehat{\theta} = e^{\Upsilon} \theta$ for $\Upsilon \in \mathscr{P}$,
the $Q$-prime curvature transforms as follows~\cite[Proposition 5.5]{Hirachi14}:
\begin{equation} \label{eq:transformation-law-of-Q-prime-curvature}
	\bm{Q}'_{\widehat{\theta}}
	= \bm{Q}'_{\theta} + 2 \bm{P}'_{\theta} \Upsilon + \bm{P}_{0,0}(\Upsilon^{2}).
\end{equation}
Here $\bm{P}'_{\theta} \colon \mathscr{P} \to \mathscr{E}(- n -1)$ is the \emph{$P$-prime operator} defined by
\begin{equation} \label{eq:P-prime-operator}
	\bm{P}'_{\theta} \Upsilon
	= - (n + 2)^{-1} [\bm{\Delta}^{n + 1}(\widetilde{\Upsilon} \log \bm{h}_{\rho})]|_{\mathcal{M}},
\end{equation}
where $\widetilde{\Upsilon}$ is a smooth extension of $\Upsilon$
that is pluriharmonic on the pseudoconvex side~\cite[Definition 4.2]{Hirachi14}.
Since $\bm{P}_{0,0}$ is formally self-adjoint
and annihilates constant functions~\cite[Proposition 5.1]{Gover-Graham05},
we have
\begin{equation}
	\int_{M} \bm{Q}'_{\widehat{\theta}} = \int_{M} \bm{Q}'_{\theta} + 2 \int_{M} \bm{P}'_{\theta} \Upsilon,
\end{equation}
for a closed CR manifold $M$.
Marugame~\cite[Theorem 1.2]{Marugame18} 
proved that $\bm{P}'_{\theta}$ is also formally self-adjoint.
Thus from $\bm{P}'_{\theta} 1 = 0$,
the integral of the $Q$-prime curvature,
the \emph{total $Q$-prime curvature},
defines a global CR invariant,
denoted by $\overline{Q}'(M)$.

\subsection{Sasakian manifolds}
\label{subsection:Sasakian-manifolds}

This subsection contains a brief summary of Sasakian manifolds from CR point of view.
See~\cite{Boyer-Galicki08} and~\cite{Sparks11}
for a comprehensive introduction to Sasakian manifolds.
Let $(S, T^{1,0}S)$ be a strictly pseudoconvex CR manifold
of dimension $2 n + 1$,
$\eta$ a contact form on $S$,
and $\xi$ the Reeb vector field with respect to $\eta$.

\begin{definition} \label{def:Sasakian-manifold}
	The triple $(S, T^{1,0}S, \eta)$ is called a \emph{Sasakian manifold}
	if the Tanaka-Webster torsion with respect to $\eta$ vanishes.
\end{definition}

For a Sasakian manifold $(S, T^{1,0} S, \eta)$,
an almost complex structure $I$ on the cone $C(S) = \mathbb{R}_{+} \times S$ of $S$
is defined by
\begin{equation}
	I( V + a \xi + b (r \partial_{r})) = J V + b \xi - a (r\partial_{r}),
\end{equation}
where $r$ is the coordinate of $\mathbb{R}_{+}$, $V \in H S$ and $a,b \in \mathbb{R}$.
Then this almost complex structure is integrable;
that is, $(C(S), I)$ is a complex manifold.
The Riemannian metric $\overline{g} = dr \otimes dr + r^{2} g_{\eta}$
is a K\"{a}hler metric on $C(S)$
and its K\"{a}hler form is equal to $dd^{c} r^{2}/2$.
Moreover, the level set $\{ r = 1 \}$ is isomorphic to $S$ as a CR manifold
and the $1$-form $\eta$ is equal to $d^{c} \log r^{2}$.

Consider the Tanaka-Webster connection with respect to $\eta$.
Note that the index $0$ is used for the component $\xi$ or $\eta$ in our index notation. 
The commutators of the derivatives for $f \in C^{\infty}(X)$ are given by
\begin{equation}
	\tensor{f}{_{[} _{\alpha} _{\beta} _{]}}
	= 0,
	\quad
	2 \tensor{f}{_{[} _{\alpha} _{\overline{\beta}} _{]}}
	= \sqrt{-1} \tensor{l}{_{\alpha} _{\overline{\beta}}} f_{0},
	\quad
	\tensor{f}{_{[} _{0} _{\alpha} _{]}}
	= 0,
\end{equation}
where $[\dotsb]$ means the antisymmetrization over the enclosed indices.
Define the \emph{Kohn Laplacian} $\Box_{b}$ and \emph{sub-Laplacian} $\Delta_{b}$ by
\begin{equation}
	\Box_{b} f = - \tensor{f}{_{\overline{\alpha}} ^{\overline{\alpha}}}, \quad
	\Delta_{b} f
	= - \tensor{f}{_{\overline{\alpha}} ^{\overline{\alpha}}}
	-  \tensor{f}{_{\alpha}^ {\alpha}}
	= \Box_{b} f + \overline{\Box}_{b} f,
\end{equation}
respectively.
From the above commutation relations,
we have
\begin{equation}
	\Box_{b} - \overline{\Box}_{b} = \sqrt{-1} n \xi.
\end{equation}
The third covariant derivatives of $f$ satisfy
\begin{align}
	\tensor{f}{_{\alpha} _{[} _{0} _{\overline{\beta}} _{]}} &= 0,
	\quad
	\tensor{f}{_{\alpha} _{[} _{0} _{\overline{\beta}} _{]}} = 0, \notag \\
	2 \tensor{f}{_{\alpha} _{[} _{\beta} _{\overline{\gamma}} _{]}}
	&= \sqrt{-1} \tensor{l}{_{\beta} _{\overline{\gamma}}} \tensor{f}{_{\alpha} _{0}}
	+ \tensor{R}{_{\alpha}^{\delta}_{\beta}_{\overline{\gamma}}} \tensor{f}{_{\delta}}.
	\label{eq:commutator-of-third-derivative}
\end{align}
From these,
it follows that
the Kohn Laplacian and sub-Laplacian commute with the Reeb vector field $\xi$.
For the proofs of these commutation relations, 
see~\cite[Lemma 2.3]{Lee88},
where these formulas are proved for general CR manifolds.

We next consider an Einstein condition for Sasakian manifolds.

\begin{definition} \label{def:Sasakian-eta-Einstein-manifold}
	Let $(S, T^{1,0}S, \eta)$ be a $(2 n + 1)$-dimensional Sasakian manifold.
	It is called a \emph{Sasakian $\eta$-Einstein manifold}
	if there exists a real constant $\lambda$ such that the Tanaka-Webster Ricci curvature
	$\tensor{\Ric}{_{\alpha}_{\overline{\beta}}}$ of $\eta$ satisfies
	\begin{equation}
		\tensor{\Ric}{_{\alpha}_{\overline{\beta}}}
		= (n + 1) \lambda \tensor{l}{_{\alpha}_{\overline{\beta}}}.
	\end{equation}
	In particular if $\lambda = 1$, it is called a \emph{Sasaki-Einstein manifold}.
	In this paper, we call the constant $(n + 1) \lambda$ the \emph{Einstein constant} of $(S, T^{1,0}S, \eta)$.
\end{definition}

There exist characterizations of Sasakian $\eta$-Einstein manifolds
in terms of $g_{\eta}$ or $\overline{g}$.

\begin{proposition}
	Let $(S, T^{1,0}S, \eta)$ be a $(2 n + 1)$-dimensional Sasakian manifold
	and $\lambda$ a real constant.
	Then the following are equivalent:
	\begin{enumerate}
		\item  \label[condition]{cond:Sasakian-eta-Einstein-manifold-1}
			$(S, T^{1,0} S, \eta)$ is a Sasakian $\eta$-Einstein manifold
			with Einstein constant $(n + 1) \lambda$;
		\item  \label[condition]{cond:Sasakian-eta-Einstein-manifold-2}
			the Ricci curvature $\Ric_{g_{\eta}}$ of $g_{\eta}$ satisfies
			\begin{equation}
				\Ric_{g_{\eta}}
				= 2((n + 1)\lambda-1) g_{\eta} + 2(n + 1)(1-\lambda) \eta \otimes \eta;
			\end{equation}
		\item  \label[condition]{cond:Sasakian-eta-Einstein-manifold-3}
			the Ricci form of $\overline{g}$ is equal to $(n + 1)(\lambda - 1) d\eta$,
			or, the Ricci curvature $\Ric_{\overline{g}}$ of $\overline{g}$ satisfies
			\begin{equation}
				\Ric_{\overline{g}}
				= 2 (n + 1) (\lambda -1) (g_{\eta} - \eta \otimes \eta).
			\end{equation}
	\end{enumerate}
\end{proposition}

\begin{proof}
	
	First, 
	we show the equivalence
	between \cref{cond:Sasakian-eta-Einstein-manifold-1,cond:Sasakian-eta-Einstein-manifold-2}.
	From \cite[Theorem 7.3.12]{Boyer-Galicki08},
	the Ricci curvature $\Ric^{TW}$ of the Tanaka-Webster connection satisfies
	\begin{equation}
		\Ric_{g_{\eta}} = \Ric^{TW} - 2g_{\eta} + 2 (n + 1) \, \eta \otimes \eta.
	\end{equation}
	Note that $\Ric$ (resp. $\Ric_{T}$) in \cite[Theorem 7.3.12]{Boyer-Galicki08}
	corresponds to $\Ric_{g_{\eta}}$ (resp. $\Ric^{TW}$) in our notation. 
	On the other hand,
	$\Ric^{TW}$ is given by
	\begin{equation} \label{eq:Ricci-curvature-of-Sasakian-manifolds}
		\Ric^{TW} = \tensor{\Ric}{_{\alpha}_{\overline{\beta}}}
		(\theta^{\alpha} \otimes \theta^{\overline{\beta}} + \theta^{\overline{\beta}} \otimes \theta^{\alpha}),
	\end{equation}
	which follows from \cref{eq:curvature-form}.
	This proves the equivalence
	between \cref{cond:Sasakian-eta-Einstein-manifold-1,cond:Sasakian-eta-Einstein-manifold-2}.

	Next,
	we show the equivalence
	between \cref{cond:Sasakian-eta-Einstein-manifold-2,cond:Sasakian-eta-Einstein-manifold-3}.
	Since $(C(S), \overline{g})$ is the warped product of $(\mathbb{R}_{+}, d r \otimes d r)$
	and $(S, g_{\eta})$ by the function $r^{2}$ on $\mathbb{R}_{+}$,
	the Ricci curvature $\Ric_{\overline{g}}$ of $\overline{g}$ satisfies
	\begin{equation}
		\Ric_{\overline{g}} = \Ric_{g_{\eta}} - 2 n g_{\eta};
	\end{equation}
	see \cite[Proposition 9.106]{Besse87} for example.
	Therefore,
	\cref{cond:Sasakian-eta-Einstein-manifold-2}
	is equivalent to \cref{cond:Sasakian-eta-Einstein-manifold-3}.
\end{proof}

Note that
\cref{cond:Sasakian-eta-Einstein-manifold-3} is equivalent to
\begin{equation} \label{eq:Sasakian-eta-Einstein-manifold}
	- dd^{c} \log \det (\partial_{a \overline{b}} r^{2})
	= (n + 1)(\lambda - 1) dd^{c} \log r^{2}
\end{equation}
for any holomorphic coordinate $(z^{1}, \dots z^{n + 1})$ of $C(S)$
since $r^{2}/2$ is a K\"{a}hler potential of $\overline{g}$.

\begin{example} \label{ex:sphere}
	Let $S^{2 n + 1} \in \mathbb{C}^{n + 1}$ be the unit sphere centered at the origin
	with the canonical CR structure,
	and $\eta_{0}$ be the contact form on $S^{2 n + 1}$ defined by
	\begin{equation}
		\eta_{0} = \frac{\sqrt{-1}}{2} \sum_{i=1}^{n + 1} (z^{i} d \bar{z}^{i} - \bar{z}^{i} dz^{i})|_{S^{2 n + 1}}.
	\end{equation}
	Then the triple $(S^{2 n + 1}, T^{1,0} S^{2 n + 1}, \eta_{0})$ is a Sasakian manifold.
	Moreover,
	the cone $(C(S^{2 n + 1}), \overline{g})$ is isomorphic to
	$(\mathbb{C}^{n + 1} \setminus \{0\}, g_{\mathrm{Euc}})$ as a K\"{a}hler manifold 
	by the map $C(S^{2 n + 1}) \ni (r, p) \mapsto r^{2} p \in \mathbb{C}^{n + 1}$.
	Here, $g_{\mathrm{Euc}}$ is the Euclidean metric on $\mathbb{C}^{n + 1}$.
	Hence the metric $\overline{g}$ is Ricci-flat,
	and $(S^{2 n + 1}, T^{1,0} S^{2 n + 1}, \eta_{0})$ is a Sasaki-Einstein manifold.
	Note that there exists a canonical projection $S^{2 n + 1} \to \mathbb{CP}^{n}$,
	and the 2-form $d \eta_{0}$ descends to
	the Fubini-Study form $\omega_{FS}$ on $\mathbb{CP}^{n}$.
\end{example}

\begin{example}
	Let $Y$ be an $n$-dimensional complex manifold,
	$L$ a holomorphic line bundle over $Y$,
	and $h$ a Hermitian metric of $L$.
	Assume that the $(1,1)$-form $\omega_{h} = d d^{c} \log h$
	defines a K\"{a}hler-Einstein metric on $Y$ with Einstein constant $(n + 1) \lambda$.
	Consider the tube $S = \{v \in L \mid h(v, v) = 1\} \subset L$
	and the contact form $\eta = d^{c} \log h$ on $S$.
	Then the triple $(S, T^{1,0}S, \eta)$ is a Sasakian $\eta$-Einstein manifold with Einstein constant $(n + 1) \lambda$.
	Note that the Reeb vector field $\xi$ with respect to $\eta$ is the generator of the $S^{1}$-action on $S$
	induced from that on $L$.
\end{example}

\section{Construction of Fefferman defining function}
\label{section:construction-of-Fefferman-defining-function}

In this section,
we construct a Fefferman defining function for Sasakian $\eta$-Einstein manifolds.
To this end, we first construct a ``good'' defining function $\rho_{S}$ of $S$ in $C(S)$.
From this defining function,
we obtain a flat Hermitian metric $\bm{h}_{S}$ of $K_{C(S)}$,
and the desired Fefferman defining function $\bm{\rho}_{S}$ is given
as the product $\rho_{S} \cdot \bm{h}_{S}^{1/(n + 2)}$.

Let $(S, T^{1,0}S, \eta)$ be
a $(2 n + 1)$-dimensional Sasakian $\eta$-Einstein manifold with Einstein constant $(n + 1) \lambda$.
We identify $S$ with the level set $\{ r = 1 \}$, 
which is a real hypersurface in a complex manifold $X = C(S)$.
Define a smooth function $\psi_{\lambda}$ on $\mathbb{R}$ by
\begin{equation}
	\psi_{\lambda}(x) =
	\begin{cases}
		\lambda^{-1} \left(\exp(\lambda x)-1\right) & \text{if} \ \lambda \neq 0, \\
		x & \text{if} \ \lambda = 0.
	\end{cases}
\end{equation}
It can be seen that
\begin{equation}
	\psi_{\lambda}' = 1 +  \lambda \psi_{\lambda}, \quad \psi_{\lambda}'' = \lambda \psi_{\lambda}',
\end{equation}
and
\begin{equation}
	\rho_{S} = \psi_{\lambda}(\log r^{2}) \in C^{\infty}(X)
\end{equation}
is a defining function of $S$ normalized by $\eta$.

\begin{proposition} \label{prop:solution-of-Fefferman-equation}
	The defining function $\rho_{S}$ satisfies the equation 
	\begin{equation}
		dd^{c} \log \mathcal{J}_{z}[\rho_{S}] = 0,
	\end{equation}
	where $z = (z^{1}, \dots , z^{n + 1})$ is a local coordinate of $X$
	and
	\begin{equation}
		\mathcal{J}_{z}[\phi]
		= - \det
		\begin{pmatrix}
			\phi & \partial \phi / \partial z^{a} \\
			\partial \phi / \partial \overline{z}^{b} & \partial^{2} \phi / \partial z^{a} \partial \overline{z}^{b}
		\end{pmatrix}
		.
	\end{equation}
\end{proposition}

\begin{proof}
	To simplify notation,
	we write $\partial_{a} = \partial / \partial z^{a}$ and
	$\partial_{a \overline{b}} = \partial^{2} / \partial z^{a} \partial \overline{z}^{b}$. 
	\begin{align}
		\mathcal{J}_{z}[\rho_{S}]
		&= - \det
		\begin{pmatrix}
			\rho_{S} & \partial_{a} \rho_{S} \\
			\partial_{\overline{b}} \rho_{S} & \partial_{a \overline{b}} \rho_{S}
		\end{pmatrix} \\
		&= - \det
		\begin{pmatrix}
			\rho_{S} &  (1 + \lambda \rho_{S}) \partial_{a} \log r^{2} \\
			\partial_{\overline{b}} \log r^{2}
			& (1 + \lambda \rho_{S}) \partial_{a \overline{b}} \log r^{2}
		\end{pmatrix} \\
		&= - (1 + \lambda \rho_{S})^{n + 1} \det
		\begin{pmatrix}
			\rho_{S} & \partial_{a} \log r^{2} \\
			\partial_{\overline{b}} \log r^{2} & \partial_{a \overline{b}} \log r^{2}
		\end{pmatrix} \\
		&= - (1 + \lambda \rho_{S})^{n + 1} r^{- 2 (n + 1)}\det
		\begin{pmatrix}
			\rho_{S} & \partial_{a} r^{2} \\
			(1+\rho_{S}) r^{-2} \partial_{\overline{b}} r^{2} & \partial_{a \overline{b}} r^{2}
		\end{pmatrix}
		.
	\end{align}
	Since $r^{2}/2$ is a K\"{a}hler potential of $\overline{g}$,
	we have
	$\partial_{a \overline{b}} r^{2} = 2 \overline{g}(\partial_{a}, \partial_{\overline{b}})$.
	Thus
	\begin{align}
		\mathcal{J}_{z}[\rho_{S}]
		&= - (1 + \lambda \rho_{S})^{n + 1} r^{- 2 (n + 1)}
		\left[\rho_{S} - (1 + \rho_{S}) (2r^{2})^{-1}\| \partial r^{2} \|_{\overline{g}}^{2} \right]
		\det (\partial_{a \overline{b}} r^{2}) \\
		&= (1 + \lambda \rho_{S})^{n + 1} r^{- 2 (n + 1)} \det (\partial_{a \overline{b}} r^{2}),
	\end{align}
	where the last equality follows from $\| \partial r^{2} \|_{\overline{g}}^{2} = 2 r^{2}$.
	Therefore from \cref{eq:Sasakian-eta-Einstein-manifold},
	\begin{align}
		- dd^{c} \log \mathcal{J}_{z}[\rho_{S}]
		&= - (n + 1) (d d^{c} \log (1 + \lambda \rho_{S}) - dd^{c} \log r^{2})
		- dd^{c} \log \det (\partial_{a \overline{b}} r^{2}) \\
		&= - (n + 1) d (\lambda d^{c} \log r^{2}) + (n + 1) \lambda d d^{c} \log r^{2} \\
		&= 0.
	\end{align}
	This proves the statement.
\end{proof}

Next,
we construct a flat Hermitian metric of $K_{X}$ by using $\rho_{S}$.

\begin{lemma} \label{lem:flat-coord}
	For each point $p \in X$,
	there exists a local coordinate $z$ near $p$
	such that $\mathcal{J}_{z}[\rho_{S}] = 1$.
	Moreover, if $w = F(z)$ is also such a local coordinate, then $\det F'$ is
	a locally constant function
	with the absolute value one,
	where $F'$ is the holomorphic Jacobian of $F$.
\end{lemma}

\begin{definition}
	A local coordinate $z$ is called a \emph{flat local coordinate} if $\mathcal{J}_{z}[\rho_{S}] = 1$ holds.
\end{definition}

\begin{proof}[Proof of \cref{lem:flat-coord}]
	Take a local coordinate $w = (w^{1}, \dots , w^{n + 1})$ near $p$.
	From \cref{prop:solution-of-Fefferman-equation},
	$\log \mathcal{J}_{w} [\rho_{S}]$ is a pluriharmonic function.
	We may assume that this is the real part of a holomorphic function $f$;
	that is,
	\begin{equation}
		\mathcal{J}_{w}[\rho_{S}] = e^{\Re f},
	\end{equation}
	if we take a sufficiently small neighborhood of $p$.
	Take a holomorphic function $g$ such that $\partial g / \partial w^{1} = e^{f/2}$.
	In general,
	for another coordinate $w' = G(w)$ of $X$, 
	\begin{equation} \label{eq:transformation-law-of-Fefferman-operator}
		\mathcal{J}_{w'}[\phi] = |\det G'|^{-2}  \mathcal{J}_{w}[\phi]
	\end{equation}
	holds.
	Thus the new local coordinate $z = (z^{1} = g(w), z^{2} = w^{2}, \dots , z^{n + 1} = w^{n + 1})$
	satisfies $\mathcal{J}_{z}[\rho_{S}] =1$.
	The second statement follows from \cref{eq:transformation-law-of-Fefferman-operator}
	and the fact that a holomorphic function with its absolute value one is locally constant.
\end{proof}

\begin{corollary} \label{cor:flat-Hermitian-metric}
	There exists the unique flat Hermitian metric $\bm{h}_{S}$ on $K_{X}$ such that
	$dz^{1} \wedge \dots \wedge dz^{n + 1}$ is
	a local orthonormal frame of $K_{X}$ for any flat local coordinate $z$,
	or equivalently, $\bm{h}_{S}$ is written as $|z^{0}|^{2 (n + 2)}$,
	where $z^{0}$ is a branched fiber coordinate
	with respect to $z$.
\end{corollary}

\begin{proposition} \label{prop:Fefferman-defining-function-of-SeE}
	The defining function
	$\bm{\rho}_{S} = \rho_{S} \cdot \bm{h}_{S}^{1/(n + 2)} \in \widetilde{\mathcal{E}}(1)$
	of $\mathcal{S} = \pi_{\mathcal{X}}^{-1}(S)$
	satisfies
	\begin{equation}
		(dd^{c} \bm{\rho}_{S})^{n + 2} = k_{n} \vol_{\mathcal{X}}.
	\end{equation}
	In particular, the obstruction function $\bm{\mathcal{O}}$ vanishes on $\mathcal{X}$.
\end{proposition}

\begin{proof}
	Take a flat local coordinate $(z^{1}, \dots, z^{n + 1})$
	and a branched fiber coordinate $z^{0}$ with respect to $z$.
	Then the volume form $\vol_{\mathcal{X}}$ is written as
	\begin{equation}
		\vol_{\mathcal{X}}
		= (\sqrt{-1})^{n + 2} (n + 2)^{2} |z^{0}|^{2 (n + 1)}
		d z^{0} \wedge d \overline{z}^{0} \wedge \dots \wedge d z^{n + 1} \wedge d \overline{z}^{n + 1}.
	\end{equation}
	On the other hand,
	the $(n + 2, n + 2)$-form $(d d^{c} \bm{\rho}_{S})^{n + 2}$ is of the form
	\begin{equation}
		(\sqrt{-1})^{n + 2} (n + 2)! \det (\partial^{2} \bm{\rho}_{S} / \partial z^{A} \partial \overline{z}^{B})
		d z^{0} \wedge d \overline{z}^{0} \wedge \dots \wedge d z^{n + 1} \wedge d \overline{z}^{n + 1}.
	\end{equation}
	Thus it suffices to show that
	\begin{equation}
		\det (\partial^{2} \bm{\rho}_{S} / \partial z^{A} \partial \overline{z}^{B}) = - |z^{0}|^{2 (n + 1)},
	\end{equation}
	which follows from the computation below:
	\begin{align}
		\det (\partial^{2} \bm{\rho}_{S} / \partial z^{A} \partial \overline{z}^{B})
		&= \det
		\begin{pmatrix}
			\rho_{S} & z^{0} \partial_{a} \rho_{S} \\
			\overline{z}^{0} \partial_{\overline{b}} \rho_{S}
			& |z^{0}|^{2} \partial_{a \overline{b}} \rho_{S}
		\end{pmatrix}
		\\
		&= |z^{0}|^{2 (n + 1)} \det
		\begin{pmatrix}
			\rho_{S} & \partial_{a} \rho_{S} \\
			\partial_{\overline{b}} \rho_{S} & \partial_{a \overline{b}} \rho_{S}
		\end{pmatrix}
		\\
		&= - |z^{0}|^{2 (n + 1)} \mathcal{J}_{z}[\rho_{S}] \\
		&= - |z^{0}|^{2 (n + 1)}.
	\end{align}
	Note that the last equality is a consequence of the definition of a flat local coordinate.
\end{proof}

\section{Proof of factorization theorem}
\label{section:proof-of-factorization-theorem}

This section is devoted to the proofs of
\cref{thm:formula-for-CR-invariant-powers-of-sub-Laplacian,thm:formula-for-P-prime-operator},
product formulas for CR invariant powers of the sub-Laplacian and the $P$-prime operator.
To prove these, introduce operators that change homogeneous degrees.

\begin{definition}
	Let $z$ be a flat local coordinate of $X = C(S)$
	and $z^{0}$ a branched fiber coordinate with respect to $z$.
	The operator $\bm{M}_{v, v'} \colon \widetilde{\mathscr{E}}(w, w')
	\to \widetilde{\mathscr{E}}(w+v, w'+v')$ is defined by
	the multiplication by $(z^{0})^{v} (\bar{z}^{0})^{v'}$.
	Note that $\bm{M}_{v, v}$ coincides with
	the multiplication by $\bm{h}_{S}^{v/(n + 2)}$.
\end{definition}

These multiplication operators $\bm{M}_{v, v'}$ induce differential operators on $S$
corresponding to $\bm{P}_{w, w'}$ and $\bm{P}'_{\eta}$.

\begin{definition} \label{def:operators-on-S}
	Let $(w, w') \in \mathbb{R}^{2}$ such that $k = w + w' + n + 1$ is a positive integer.
	A differential operator $P_{w, w'}$ on $C^{\infty}(S)$ is defined by
	\begin{equation}
		P_{w, w'} = \bm{M}_{k - w, k - w'} \circ \bm{P}_{w, w'} \circ \bm{M}_{w, w'}.
	\end{equation}
	Similarly, an operator $P'_{\eta} \colon \mathscr{P} \to C^{\infty}(S)$
	is defined by
	\begin{equation}
		P'_{\eta} = \bm{M}_{n + 1, n + 1} \circ \bm{P}'_{\eta}.
	\end{equation}
	These are independent of the choice of a flat local coordinate and a branched fiber coordinate.
\end{definition}

In the following,
we use the ambient metric $\bm{g} = \bm{g}[\bm{\rho}_{S}]$
for $\bm{\rho}_{S}$ defined in \cref{prop:Fefferman-defining-function-of-SeE}.
The most important ingredient for the proofs of
\cref{thm:formula-for-CR-invariant-powers-of-sub-Laplacian,thm:formula-for-P-prime-operator} is the following

\begin{proposition} \label{prop:power-of-Laplacian}
	\begin{equation} \label{eq:power-of-Laplacian1}
		\bm{\Delta}^{k}
		= \bm{M}_{-k-1, 0} (\bm{M}_{2, 0} \bm{\Delta})^{k} \bm{M}_{-k+1, 0}.
	\end{equation}
	If $k \geq 2$, then
	\begin{equation} \label{eq:power-of-Laplacian2}
		\bm{\Delta}^{k}
		= \bm{M}_{-1, -1} \bm{\Delta}^{k-2} \bm{M}_{0, k-1}
		\bm{\Delta} \bm{M}_{k, - k + 2} \bm{\Delta} \bm{M}_{-k + 1, 0}.
	\end{equation}
\end{proposition}

We will give the proof of this proposition later. 
We need additionally the following lemma to obtain the factorization formula
for the $P$-prime operator.

\begin{lemma} \label{lem:boundary-value-of-intermediate-operator}
	Let $\Upsilon$ be a CR pluriharmonic function on $S$
	and $\widetilde{\Upsilon}$ its pluriharmonic extension.
	Then the function $\bm{I} \widetilde{\Upsilon}$ defined by
	\begin{equation}
		\bm{I} \widetilde{\Upsilon}
		= \bm{M}_{0, n} \bm{\Delta} \bm{M}_{n + 1, - n + 1} \bm{\Delta} \bm{M}_{- n, 0}
		(\widetilde{\Upsilon} \log |z^{0}|^{2})
	\end{equation}
	is an element of $\widetilde{\mathscr{E}}(-1)$
	modulo a term that vanishes to infinite order at $\mathcal{S}$.
\end{lemma}

The proof of \cref{lem:boundary-value-of-intermediate-operator} 
is delayed to the end of this section.
Computations of the homogeneous degrees in
\cref{eq:power-of-Laplacian1,eq:power-of-Laplacian2} give the following

\begin{corollary}
	The operator $\bm{P}_{w, w'}$ for $k = w + w' + n + 1$ has the formula
	\begin{equation}
		\bm{P}_{w, w'} = \bm{M}_{-k-1, 0} \prod_{j=0}^{k-1}
		(\bm{M}_{2,0} \bm{L}_{w'-w+k-2j-1}) \bm{M}_{-k+1, 0},
	\end{equation}
	where
	\begin{equation}
		\bm{L}_{\mu}
		= \bm{P}_{\frac{- \mu - n}{2}, \frac{\mu - n}{2}}.
	\end{equation}
	Similarly, the $P$-prime operator $\bm{P}'_{\eta}$ is written as
	\begin{equation}
		\bm{P}'_{\eta} \Upsilon
		= - \bm{M}_{-1, -1} \bm{P}_{-1, -1} [(\bm{I} \widetilde{\Upsilon})|_{\mathcal{S}}].
	\end{equation}
\end{corollary}

From this observation,
it suffices to prove \cref{prop:power-of-Laplacian},
\cref{lem:boundary-value-of-intermediate-operator},
\begin{equation} \label{eq:L-operators}
	L_{\mu} = \bm{M}_{\frac{\mu + n + 2}{2}, \frac{- \mu + n + 2}{2}}
	\circ \bm{L}_{\mu} \circ \bm{M}_{\frac{- \mu - n}{2}, \frac{\mu - n}{2}},
\end{equation}
and
\begin{equation} \label{eq:I-operators}
	\bm{I} \widetilde{\Upsilon}|_{\mathcal{S}}
	= - n^{-1} \bm{M}_{-1, -1} (\Delta_{b}^{2} + n^{2} \lambda \Delta_{b}) \Upsilon
\end{equation}
for the proofs of \cref{thm:formula-for-CR-invariant-powers-of-sub-Laplacian,thm:formula-for-P-prime-operator}.

To show \cref{prop:power-of-Laplacian},
define a differential operator $\bm{C}$ by
\begin{equation}
	\bm{C} = [\bm{\Delta}, \bm{M}_{1, 0}].
\end{equation}
Key commutation relations are the following

\begin{lemma}
	\begin{equation}
		[\bm{\Delta}, \bm{M}_{v, 0}] = v \bm{M}_{v-1, 0} \bm{C},
		\quad [\bm{C}, \bm{M}_{v, v'}] = \lambda v' \bm{M}_{v, v'-1},
		\quad [\bm{\Delta}, \bm{C}] = 0.
	\end{equation}
\end{lemma}

\begin{proof}
	Define a $(1, 0)$-vector field $Z_{n + 1}$ on $X$ by
	\begin{equation}
		Z_{n + 1} = \frac{1}{2} (r \partial_{r} - \sqrt{-1} \xi).
	\end{equation}
	A direct computation shows that 
	$[Z_{n + 1}, \overline{W}] \in \Gamma(T^{0, 1} X)$
	for any $\overline{W} \in \Gamma(T^{0, 1} X)$,
	which implies that $Z_{n + 1}$ is a holomorphic vector field.
	A local frame $(Z_{\alpha})$ of $T^{1, 0} S$ induces
	the frame $(Z_{0}, Z_{\alpha}, Z_{n + 1})$ of $T^{1, 0} \mathcal{X}$.
	With this frame, the matrix representations of $\bm{g}$ and its inverse are given by
	\begin{equation}
		\tensor{\bm{g}}{_{A}_{\overline{B}}}
		= |z^{0}|^{2}
		\begin{pmatrix}
			\rho_{S} & 0 & 1 + \lambda \rho_{S} \\
			0 & (1 + \lambda \rho_{S}) \tensor{l}{_{\alpha} _{\overline{\beta}}} & 0 \\
			1 + \lambda \rho_{S} & 0 & \lambda (1 + \lambda \rho_{S})
		\end{pmatrix}
		,
	\end{equation}
	and
	\begin{equation} \label{eq:matrix-representation}
		\tensor{\bm{g}}{^{A}^{\overline{B}}} = |z^{0}|^{-2}
		\begin{pmatrix}
			- \lambda & 0 & 1 \\
			0 & (1 + \lambda \rho_{S})^{-1} \tensor{l}{^{\alpha} ^{\overline{\beta}}} & 0 \\
			1 & 0 & -\rho_{S} (1 + \lambda \rho_{S})^{-1}
		\end{pmatrix}
		.
	\end{equation}
	
	Denote by $f$ the holomorphic function $z^{0}$ for simplicity.
	Since the Laplacian is of the form
	$- \tensor{\bm{\nabla}}{_{A}} \tensor{\bm{\nabla}}{^{A}}$,
	\begin{equation}
		\bm{C} = - \tensor{f}{_{A}} \tensor{\bm{\nabla}}{^{A}}.
	\end{equation}
	Hence
	\begin{equation}
		[\bm{\Delta}, \bm{M}_{v ,0}]
		= - v f^{v-1} \tensor{f}{_{A}} \tensor{\nabla}{^{A}}
		= v \bm{M}_{v-1, 0} \bm{C},
	\end{equation}
	and
	\begin{equation}
		[\bm{C}, \bm{M}_{v, v'}]
		= - v' f^{v} \overline{f}^{v'-1} \tensor{f}{_{A}} \tensor{\overline{f}}{^{A}}
		= \lambda v' \bm{M}_{v, v'-1}.
	\end{equation}
	Here, we use the fact that $\tensor{f}{_{A}} \tensor{\overline{f}}{^{A}} = - \lambda$,
	which follows from \cref{eq:matrix-representation}.
	Similarly,
	\begin{align}
		\bm{\Delta} \bm{C}
		&= \tensor{\bm{\nabla}}{_{B}} (\tensor{f}{_{A}}
		\tensor{\bm{\nabla}}{^{B}} \tensor{\bm{\nabla}}{^{A}}) \\
		&= \tensor{f}{_{A}_{B}} \tensor{\bm{\nabla}}{^{A}} \tensor{\bm{\nabla}}{^{B}}
		+ \tensor{f}{_{A}} \tensor{\bm{\nabla}}{_{B}} \tensor{\bm{\nabla}}{^{B}} \tensor{\bm{\nabla}}{^{A}} \\
		&= \tensor{f}{_{A}_{B}} \tensor{\bm{\nabla}}{^{A}} \tensor{\bm{\nabla}}{^{B}}
		+ \bm{C} \bm{\Delta};
	\end{align}
	the last equality holds since $\bm{g}$ is Ricci-flat.
	Thus, it is sufficient to show that $\tensor{f}{_{A}_{B}} = 0$.
	From definition,
	\begin{equation}
		\tensor{f}{_{A}_{B}} = Z_{B} Z_{A} f - (\bm{\nabla}_{Z_{B}} Z_{A}) f,
	\end{equation}
	and
	\begin{equation}
		\bm{g}(\bm{\nabla}_{Z_{B}}Z_{A}, Z_{\overline{C}})
		= Z_{B} (\bm{g}(Z_{A}, Z_{\overline{C}}))
		- \bm{g}(Z_{A}, [Z_{B}, Z_{\overline{C}}]).
	\end{equation}
	Since $f = z^{0}$,
	we need only to consider the $Z_{0}$-component of $\bm{\nabla}_{Z_{B}} Z_{A}$.
	Thus it is enough to compute the value
	$\bm{g}(\bm{\nabla}_{Z_{B}}Z_{A}, Z_{\overline{C}})$ for $C = 0$ or $n + 1$
	from the matrix representation of $\bm{g}$.
	In this case,
	\begin{equation}
		\bm{g}(\bm{\nabla}_{Z_{B}}Z_{A}, Z_{\overline{C}}) = Z_{B} (\bm{g}(Z_{A}, Z_{\overline{C}})),
	\end{equation}
	 since the $(0,1)$-vector fields $Z_{\overline{0}}$ and $Z_{\overline{n + 1}}$ are anti-holomorphic.
	 Under these observations,
	 a direct calculation shows $Z_{B} Z_{A} f = (\bm{\nabla}_{Z_{B}} Z_{A}) f$.
\end{proof}

The following lemma is a consequence of the above commutation relations.

\begin{lemma}
	\begin{equation}
		\bm{M}_{-1, 0} \bm{\Delta}^{k} \bm{M}_{1, 0}
		= \bm{\Delta}^{k} + k \bm{M}_{-1, 0} \bm{\Delta}^{k-1} \bm{C}.
	\end{equation}
\end{lemma}

\begin{proof}
	\begin{align}
		\bm{M}_{-1, 0} \bm{\Delta}^{k} \bm{M}_{1, 0}
		&= \bm{\Delta}^{k} + \sum_{j=0}^{k-1} \bm{M}_{-1, 0} \bm{\Delta}^{k-j-1}
		[\bm{\Delta}, \bm{M}_{1, 0}] \bm{\Delta}^{j} \\
		&= \bm{\Delta}^{k}
		+ \sum_{j=0}^{k-1} \bm{M}_{-1, 0} \bm{\Delta}^{k-j-1} \bm{C} \bm{\Delta}^{j} \\
		&= \bm{\Delta}^{k} + k \bm{M}_{-1, 0} \bm{\Delta}^{k-1} \bm{C}.
	\end{align}
	This proves the statement.
\end{proof}

\begin{proof}[Proof of \cref{prop:power-of-Laplacian}]
	We prove \cref{eq:power-of-Laplacian1,eq:power-of-Laplacian2}
	by induction in $k$.
	First we prove \cref{eq:power-of-Laplacian1}.
	The case $k=1$ is trivial.
	Assume that the formula holds for $k$.
	Then,
	\begin{align}
		&\bm{M}_{-k-2, 0} (\bm{M}_{2, 0} \bm{\Delta})^{k+1} \bm{M}_{-k, 0} \\
		&= \bm{M}_{-1, 0} [\bm{M}_{-k-1, 0} (\bm{M}_{2, 0} \bm{\Delta})^{k} \bm{M}_{-k+1, 0}]
		[\bm{M}_{k+1, 0} \bm{\Delta} \bm{M}_{-k, 0}] \\
		&= \bm{M}_{-1, 0} \bm{\Delta}^{k} (\bm{M}_{1, 0} \bm{\Delta} - k \bm{C}) \\
		&= (\bm{\Delta}^{k} + k \bm{M}_{-1, 0} \bm{\Delta}^{k-1} \bm{C}) \bm{\Delta}
		- k \bm{M}_{-1, 0} \bm{\Delta}^{k} \bm{C} \\
		&= \bm{\Delta}^{k+1}.
	\end{align}
	This proves the formula for $k+1$.
	Next, consider \cref{eq:power-of-Laplacian2}.
	If $k=2$,
	\begin{align}
		\bm{M}_{-1, -1} \bm{M}_{0, 1} \bm{\Delta}
		\bm{M}_{2, 0} \bm{\Delta} \bm{M}_{-1, 0}
		&= \bm{M}_{-1, 0} \bm{\Delta} \bm{M}_{1, 0} \bm{\Delta}
		- \bm{M}_{-1, 0} \bm{\Delta} \bm{C} \\
		&= \bm{\Delta}^{2} + \bm{M}_{-1, 0} \bm{C} \bm{\Delta} - \bm{M}_{-1, 0} \bm{\Delta} \bm{C} \\
		&= \bm{\Delta}^{2}.
	\end{align}
	Assume that the formula holds for $k$.
	Then,
	\begin{align}
		& \bm{M}_{-1, -1} \bm{\Delta}^{k-1} \bm{M}_{0, k} \bm{\Delta}
		\bm{M}_{k+1, -k+1} \bm{\Delta} \bm{M}_{-k, 0} \\
		&= \bm{M}_{-1, -1} \bm{\Delta}^{k-2} [\bm{\Delta}, \bm{M}_{0, k-1}] \bm{M}_{0,1} \bm{\Delta}
		\bm{M}_{k+1, -k+1} \bm{\Delta} \bm{M}_{-k, 0} \\
		&\quad + \bm{M}_{-1, -1} \bm{\Delta}^{k-2} \bm{M}_{0, k-1} \bm{\Delta} \bm{M}_{0, 1}
		[\bm{\Delta}, \bm{M}_{0, -k+1}] \bm{M}_{k + 1, 0} \bm{\Delta} \bm{M}_{-k, 0} \\
		&\quad + \bm{M}_{-1, -1} \bm{\Delta}^{k-2} \bm{M}_{0, k-1} \bm{\Delta} \bm{M}_{0, -k+2}
		[\bm{\Delta}, \bm{M}_{k, 0}] \bm{M}_{1, 0} \bm{\Delta} \bm{M}_{-k, 0} \\
		&\quad + \bm{M}_{-1, -1} \bm{\Delta}^{k-2} \bm{M}_{0, k-1}
		\bm{\Delta} \bm{M}_{k, -k+2} \bm{\Delta} \bm{M}_{1, 0} [\bm{\Delta}, \bm{M}_{-k, 0}] \\
		&\quad + \bm{M}_{-1, -1} \bm{\Delta}^{k-2} \bm{M}_{0, k-1}
		\bm{\Delta} \bm{M}_{k, -k+2} \bm{\Delta} \bm{M}_{-k+1,0} \bm{\Delta} \\
		&= \bm{\Delta}^{k+1}.
	\end{align}
	Note that
	in the last equality,
	the first and the second term, and the third and the fourth term
	cancel respectively. 
	This proves \cref{eq:power-of-Laplacian2} for $k+1$.
\end{proof}

\begin{proof}[Proof of \cref{thm:formula-for-CR-invariant-powers-of-sub-Laplacian}]
	What is left is to prove \cref{eq:L-operators}.
	Since $\bm{\Delta} = - \tensor{\bm{g}}{^{A}^{\overline{B}}} \tensor{\bm{\nabla}}{_{A}} \tensor{\bm{\nabla}}{_{\overline{B}}}$,
	we need only to consider $\partial \overline{\partial} \widetilde{\bm{f}} (Z_{A}, Z_{\overline{B}})$
	for $\widetilde{\bm{f}} \in \widetilde{\mathscr{E}}(w, w')$
	and $(A, \overline{B})$ with $\tensor{\bm{g}}{^{A}^{\overline{B}}} \neq 0$.
	Since $Z_{0}$ and $Z_{n + 1}$ are holomorphic vector fields,
	\begin{align}
		\partial \overline{\partial} \widetilde{\bm{f}} (Z_{0}, Z_{\overline{0}})
		&= w w' \widetilde{\bm{f}} ,
		& \partial \overline{\partial} \widetilde{\bm{f}} (Z_{0}, Z_{\overline{n + 1}})
		&= w Z_{\overline{n + 1}} \widetilde{\bm{f}} ,\\
		\partial \overline{\partial} \widetilde{\bm{f}} (Z_{n + 1}, Z_{\overline{0}})
		&= w' Z_{n + 1} \widetilde{\bm{f}} ,
		&\partial \overline{\partial} \widetilde{\bm{f}} (Z_{n + 1}, Z_{\overline{n + 1}})
		&= \frac{1}{2} (Z_{n + 1} Z_{\overline{n + 1}} + Z_{\overline{n + 1}} Z_{n + 1}) \widetilde{\bm{f}} .
	\end{align}
	On the other hand,
	the commutator $[Z_{\alpha}, Z_{\overline{\beta}}]$ is equal to
	\begin{equation}
		\nabla_{Z_{\alpha}} Z_{\overline{\beta}} - \nabla_{Z_{\overline{\beta}}} Z_{\alpha}
		+ \tensor{l}{_{\alpha}_{\overline{\beta}}} (Z_{n + 1} - Z_{\overline{n + 1}}).
	\end{equation}
	Thus we have
	\begin{align}
		\partial \overline{\partial} \widetilde{\bm{f}} (Z_{\alpha}, Z_{\overline{\beta}})
		&= Z_{\alpha} Z_{\overline{\beta}} \widetilde{\bm{f}}
		- \overline{\partial} \widetilde{\bm{f}}([Z_{\alpha}, Z_{\overline{\beta}}]) \\
		&= \frac{1}{2} (Z_{\alpha} Z_{\overline{\beta}} + Z_{\overline{\beta}} Z_{\alpha}) \widetilde{\bm{f}}
		+ \frac{1}{2} \partial \widetilde{\bm{f}}([Z_{\alpha}, Z_{\overline{\beta}}])
		- \frac{1}{2} \overline{\partial} \widetilde{\bm{f}}([Z_{\alpha}, Z_{\overline{\beta}}]) \\
		&= \frac{1}{2}(Z_{\alpha} Z_{\overline{\beta}} - \nabla_{Z_{\alpha}} Z_{\overline{\beta}}
		+ Z_{\overline{\beta}} Z_{\alpha} - \nabla_{Z_{\overline{\beta}}} Z_{\alpha}) \widetilde{\bm{f}}
		+ \frac{1}{2} \tensor{l}{_{\alpha}_{\overline{\beta}}} (Z_{n + 1} + Z_{\overline{n + 1}}) \widetilde{\bm{f}}.
	\end{align}
	Hence for $\widetilde{f} \in C^{\infty}(X)$,
	\begin{equation} \label{eq:ambient-laplacian}
		\begin{split}
			&(\bm{M}_{1 - w, 1- w'} \circ \bm{\Delta} \circ \bm{M}_{w, w'} \widetilde{f})|_{\mathcal{S}} \\
			&= \left(\frac{1}{2} \Delta_{b} + \lambda w w' \right) \widetilde{f} |_{S}
			+ \left(- \frac{n}{2} -w'\right) (Z_{n + 1} \widetilde{f}) |_{S} \\
			& \quad + \left(- \frac{n}{2} - w\right) (Z_{\overline{n + 1}} \widetilde{f}) |_{S}.
		\end{split}
	\end{equation}
	In particular,
	taking $(w, w') = ((- \mu - n) / 2, (\mu - n) / 2)$,
	we have
	\begin{align}
		\bm{M}_{1 - w, 1- w'} \circ \bm{L}_{\mu} \circ \bm{M}_{w, w'}
		&= \frac{1}{2} \Delta_{b} + \frac{\sqrt{-1}}{2}\mu \xi
		+ \frac{1}{4} \lambda (n -\mu)(n + \mu) \\
		&= L_{\mu}.
	\end{align}
	This proves the theorem.
\end{proof}

\begin{proof}[Proof of \cref{lem:boundary-value-of-intermediate-operator,thm:formula-for-P-prime-operator}]
	It suffices to show \cref{lem:boundary-value-of-intermediate-operator,eq:I-operators}.
	First, note that $Z_{n + 1} Z_{\overline{n + 1}} \widetilde{\Upsilon}$ and $\bm{\Delta} \widetilde{\Upsilon}$
	vanish to infinite order at the boundary.
	In particular, \cref{eq:ambient-laplacian} for $(w, w') = (0, 0)$ gives that
	\begin{equation}
		[(Z_{n + 1} + Z_{\overline{n + 1}}) \widetilde{\Upsilon}]|_{S} = n^{-1} \Delta_{b} \Upsilon.
	\end{equation}
	In the following, we compute modulo functions that vanish to infinite order at the boundary.
	\begin{align}
		&\bm{\Delta} \bm{M}_{- n, 0} (\widetilde{\Upsilon} \log |z^{0}|^{2}) \\
		&= - \langle d \widetilde{\Upsilon}, d ((z^{0})^{- n} \log |z^{0}|^{2}) \rangle_{\bm{g}}
		+ \widetilde{\Upsilon} \bm{\Delta} ((z^{0})^{- n} \log |z^{0}|^{2}) \\
		&= \bm{M}_{- n -1, -1} [n (Z_{\overline{n + 1}} \widetilde{\Upsilon}) \log |z^{0}|^{2}
		-(Z_{n + 1} + Z_{\overline{n + 1}}) \widetilde{\Upsilon} - n \lambda \widetilde{\Upsilon}].
	\end{align}
	Here $\langle \cdot, \cdot \rangle_{\bm{g}}$ is the inner product on $T^{*} \mathcal{X}$
	induced from $\bm{g}$.
	Hence
	\begin{align}
		&\bm{M}_{0, n} \bm{\Delta} \bm{M}_{n + 1, - n + 1}
		\bm{\Delta} \bm{M}_{- n, 0} (\widetilde{\Upsilon} \log |z^{0}|^{2}) \\
		&= \bm{M}_{0, n} \bm{\Delta} \bm{M}_{0, -n}
		[n (Z_{\overline{n + 1}} \widetilde{\Upsilon}) \log |z^{0}|^{2}
		-(Z_{n + 1} + Z_{\overline{n + 1}}) \widetilde{\Upsilon} - n \lambda \widetilde{\Upsilon}] \\
		&= \bm{M}_{-1, -1} [- n (Z_{n + 1}^{2} + Z_{\overline{n + 1}}^{2})\widetilde{\Upsilon}
		- n^{2} \lambda (Z_{n + 1} + Z_{\overline{n + 1}})\widetilde{\Upsilon}] \\
		&= \bm{M}_{-1, -1} [- n (Z_{n + 1} - Z_{\overline{n + 1}})^{2} \widetilde{\Upsilon}
		- n^{2} \lambda (Z_{n + 1} + Z_{\overline{n + 1}})\widetilde{\Upsilon}] \\
		&= \bm{M}_{-1, -1} [n \xi^{2} \widetilde{\Upsilon}
		- n^{2} \lambda (Z_{n + 1} + Z_{\overline{n + 1}})\widetilde{\Upsilon}],
	\end{align}
	which is an element of $\widetilde{\mathscr{E}}(-1)$.
	Moreover,
	on $\mathcal{S}$,
	\begin{equation}
		(\bm{I} \widetilde{\Upsilon})|_{\mathcal{S}}
		= - n^{-1} |z^{0}|^{-2} (\Delta_{b}^{2} + n^{2} \lambda \Delta_{b}) \Upsilon.
	\end{equation}
	Here we use the fact that $\Delta_{b}^{2} + n^{2} \xi^{2}$ annihilates CR pluriharmonic functions 
	on Sasakian manifolds; 
	see~\cite[Proposition 3.2]{Graham-Lee88}.
\end{proof}

\section{Variation of total \texorpdfstring{$Q$}{Q}-prime curvature}
\label{section:variation-of-total-Q-prime-curvature}

In this section,
we consider the first and the second variation of the total $Q$-prime curvature.

We recall a variational formula for the total $Q$-prime curvature
under deformations of real hypersurfaces.
Let $(M_{t})_{t \in (-1, 1)}$ be
a smooth family of closed strictly pseudoconvex real hypersurfaces
in a complex manifold $X$.
Take a Fefferman defining function $\bm{\rho}_{t}$ of $\mathcal{M}_{t} = \pi_{\mathcal{X}}^{-1} (M_{t})$
such that it is smooth in the parameter $t \in (-1, 1)$.
Assume that there exists a flat Hermitian metric $\bm{h}$ of $K_{X}$ near $M = M_{0}$.
In this setting, the function $\rho_{t} = \bm{\rho}_{t} \cdot \bm{h}^{-1/(n + 2)}$ is a defining function of $M_{t}$,
and the corresponding contact form $\theta_{t} = d^{c} \rho_{t}|_{M_{t}}$ is pseudo-Einstein.

\begin{theorem}[{\cite[Theorem 1.2]{Hirachi-Marugame-Matsumoto17}}]
	Under the above assumptions, the total $Q$-prime curvature satisfies 
	\begin{equation} \label{eq:first-variation}
		\left. \frac{d}{dt} \right|_{t=0} \overline{Q}'(M_{t})
		= c_{n} \int_{M} \bm{\varphi} \bm{\mathcal{O}},
	\end{equation}
	where $\bm{\mathcal{O}}$ is the obstruction function of $M$,
	$\bm{\varphi} = (d/dt)|_{t = 0} \bm{\rho}_{t}|_{\mathcal{M}} \in \mathscr{E}(1)$,
	and $c_{n} = 2 n! (n + 2)!$.
	Moreover if the obstruction function of $M$ vanishes,
	then 
	\begin{equation} \label{eq:second-variation}
		\left. \frac{d^{2}}{dt^{2}} \right|_{t=0} \overline{Q}'(M_{t})
		= c'_{n} \int_{M} \bm{\varphi} (\bm{R} \bm{\varphi}),
	\end{equation}
	where $\bm{R}$ is as in \cref{lem:super-critical-GJMS-operator} and $c'_{n} = -2 ((n + 1)(n + 2))^{-1}$.
\end{theorem}

In the following, 
we consider a closed Sasakian $\eta$-Einstein manifold $S$
and a smooth family $(M_{t})_{t \in (-1, 1)}$ of closed real hypersurfaces in $X = C(S)$ such that $M_{0} = S$.

The first variation of the total $Q$-prime curvature
can be computed from \cref{eq:first-variation}.

\begin{proof}[Proof of \cref{prop:first-variation-for-SeE}]
	Since $K_{X}$ has a flat Hermitian metric $\bm{h}_{S}$,
	we can apply \cref{eq:first-variation}.
	Thus \cref{prop:first-variation-for-SeE} follows from the vanishing of the obstruction function.
\end{proof}

Next,
consider the second variation of the total $Q$-prime curvature.
Take a Fefferman defining function $\bm{\rho}_{t}$ of $\mathcal{M}_{t}$
that is smooth in $t$ and coincides with $\bm{\rho}_{S}$
constructed in \cref{prop:Fefferman-defining-function-of-SeE}
at $t = 0$.
Then the second variation of $\overline{Q}'(M_{t})$ satisfies
\begin{equation}
	\left. \frac{d^{2}}{dt^{2}} \right|_{t=0} \overline{Q}'(M_{t})
	= c'_{n} \int_{M} \varphi (P_{1, 1} \varphi) \, \eta \wedge (d\eta)^{n},
\end{equation}
where $\varphi = (d/dt)|_{t = 0} (\bm{\rho}_{t} \cdot \bm{h}_{S}^{-1/(n + 2)})|_{\mathcal{S}} \in C^{\infty}(S)$.
Note that the constant $c'_{n}$ is negative.
Hence it is enough to study spectral properties of $P_{1, 1}$
for the proofs of properties of the second variation.

Before studying $P_{1, 1}$,
we consider a relation between $D_{\eta}$ introduced in \cref{subsection:deformation-of-CR-structures}
and $L_{\mu}$'s.

\begin{lemma} \label{lem:kernel-of-edge-operator}
	The operator $4 D_{\eta}^{*} D_{\eta}$
	coincides with $L_{n + 2} L_{n}$.
	In particular, the operator $L_{n + 2} L_{n}$ is a non-negative operator
	and its kernel coincides with $\ker D_{\eta}$ if $S$ is closed.
	Similarly, the operator $L_{- n - 2} L_{- n}$ is non-negative
	and its kernel is equal to that of $\overline{D}_{\eta}$.
\end{lemma}

\begin{proof}
	Since $L_{- n - 2} L_{- n}$ is the complex conjugate of $L_{n + 2} L_{n}$,
	it suffices to prove the lemma for $D_{\eta}$ and $L_{n + 2} L_{n}$.
	Since the Tanaka-Webster torsion for $\eta$ vanishes,
	the operator $D_{\eta}$ is given by
	$2 \tensor{(D_{\eta} F)}{_{\overline{\alpha}} ^{\beta}} = \tensor{F}{_{\overline{\alpha}} ^{\beta}}$.
	Hence $4 D_{\eta}^{*} D_{\eta} F = \tensor{F}{_{\overline{\alpha}} ^{\beta} ^{\overline{\alpha}} _{\beta}}$.
	From \cref{eq:commutator-of-third-derivative},
	\begin{align}
		\tensor{F}{_{\overline{\alpha}} ^{\beta} ^{\overline{\alpha}}}
		&= \tensor{F}{_{\overline{\alpha}} ^{\overline{\alpha}} ^{\beta}}
		- \sqrt{-1} \tensor{F}{^{\beta} _{0}}
		+ \tensor{R}{_{\overline{\alpha}} ^{\overline{\delta}} ^{\beta} ^{\overline{\alpha}}}
		\tensor{F}{_{\overline{\delta}}} \\
		&= \tensor{F}{_{\overline{\alpha}} ^{\overline{\alpha}} ^{\beta}}
		- \sqrt{-1} \tensor{F}{^{\beta} _{0}}
		+ (n + 1) \lambda \tensor{F}{^{\beta}};
	\end{align}
	in the last equality,
	we use the fact that
	\begin{equation}
		\tensor{R}{_{\overline{\alpha}} ^{\overline{\delta}} ^{\beta} ^{\overline{\alpha}}}
		= \tensor{\Ric}{^{\overline{\delta}} ^{\beta}}
		= (n + 1)\lambda \tensor{l}{^{\beta} ^{\overline{\delta}}}.
	\end{equation}
	Thus
	\begin{align}
		\tensor{F}{_{\overline{\alpha}} ^{\beta} ^{\overline{\alpha}} _{\beta}}
		&= \tensor{F}{_{\overline{\alpha}} ^{\overline{\alpha}} ^{\beta} _{\beta}}
		- \sqrt{-1} \tensor{F}{^{\beta} _{0} _{\beta}}
		+ (n + 1) \lambda \tensor{F}{^{\beta} _{\beta}} \\
		&= \Box_{b}^{2} F + \sqrt{-1} \xi \Box_{b} F - (n + 1) \lambda \Box_{b} F \\
		&= L_{n + 2} L_{n} F.
	\end{align}
	In particular if $S$ is closed,
	$L_{n + 2} L_{n}$ is non-negative, and its kernel coincides with $\ker D_{\eta}$.
\end{proof}

We rewrite \cref{lem:kernel-of-edge-operator}
by using the spectral theory of the sub-Laplacian and Reeb vector field.
Since the sub-Laplacian $\Delta_{b}$ is
a non-negative subelliptic self-adjoint operator,
its spectrum $\sigma(\Delta_{b})$ is a discrete subset of $[0, \infty)$,
consists only of eigenvalues,
and the eigenspace $\mathscr{H}_{p}$ with eigenvalue $p \in \sigma(\Delta_{b})$
is a finite-dimensional subspace in $C^{\infty}(S)$.
Note that $\mathscr{H}_{0} = \mathbb{C}$.
Moreover,
the vector field $\sqrt{-1} \xi$ is self-adjoint
and commutes with the sub-Laplacian.
Hence each eigenspace $\mathscr{H}_{p}$
is decomposed into the orthogonal direct sum of $\mathscr{H}_{p, q}$,
where $q$ is a real number and
\begin{equation}
	\mathscr{H}_{p, q} = \{ f \in \mathscr{H}_{p} \mid \sqrt{-1} \xi f = q f \}.
\end{equation}
An important fact is that $p \geq n |q|$ if  $\mathscr{H}_{p, q} \neq 0$.
This is because the operators
$2 \Box_{b} = \Delta_{b} + \sqrt{-1} n \xi$
and $2 \overline{\Box}_{b} = \Delta_{b} - \sqrt{-1} n \xi$
are non-negative operators.
On $\mathscr{H}_{p, q}$,
the operator $L_{\mu}$ coincides with the multiplication by
\begin{equation}
	\frac{1}{2} p + \frac{1}{2} \mu q + \frac{1}{4} \lambda (n - \mu)(n + \mu).
\end{equation}
From this point of view,
\cref{lem:kernel-of-edge-operator} states that
if $\mathscr{H}_{p, q} \neq 0$,
then 
\begin{equation} \label{eq:non-negativity-of-edge-operator1}
	\left(\frac{1}{2}p + \frac{1}{2} n q \right)\left(\frac{1}{2}p + \frac{1}{2} (n + 2) q - (n + 1) \lambda \right) \geq 0
\end{equation}
and
\begin{equation} \label{eq:non-negativity-of-edge-operator2}
	\left(\frac{1}{2}p - \frac{1}{2} n q \right)\left(\frac{1}{2}p - \frac{1}{2} (n + 2) q - (n + 1) \lambda \right) \geq 0,
\end{equation}
and moreover the equality holds if and only if $\mathscr{H}_{p, q}$ is contained in $\ker D_{\eta}$
(resp.\ $\ker \overline{D}_{\eta}$).

Now we return the study of $P_{1, 1}$.
From \cref{thm:formula-for-CR-invariant-powers-of-sub-Laplacian},
$P_{1, 1}$ coincides with the multiplication by
\begin{equation}
	\prod_{j=0}^{n + 2} \left(\frac{1}{2} p + \frac{1}{2} (n + 2 - 2 j) q + \lambda(j - 1)(n + 1 - j) \right)
\end{equation}
on $\mathscr{H}_{p, q}$.
Equations \cref{eq:non-negativity-of-edge-operator1,eq:non-negativity-of-edge-operator2}
give that
the quantity
\begin{equation} \label{eq:non-negativity-of-edge-operator3}
	\prod_{j=0, 1, n + 1, n + 2} \left(\frac{1}{2} p + \frac{1}{2} (n + 2 - 2 j) q + \lambda(j - 1)(n + 1 - j) \right)
\end{equation}
is non-negative for $\mathscr{H}_{p, q} \neq 0$,
and equal to zero
if and only if $\mathscr{H}_{p, q}$ is contained in
$\ker D_{\eta} + \ker \overline{D}_{\eta}$.
Thus if $n \geq 2$,
the sign of the second variation depends on that of
\begin{equation} \label{eq:middle-operator}
	\prod_{j=2}^{n} \left(\frac{1}{2} p + \frac{1}{2} (n + 2 - 2 j) q + \lambda(j - 1)(n + 1 - j) \right).
\end{equation}

\begin{proof}[Proof of \cref{thm:second-variation-for-non-negative-SeE}]
	If $n = 1$,
	the theorem follows from the above argument. 
	In the following, we consider the case $n \geq 2$.
	Let $2 \leq j \leq n$.
	If $\lambda \geq 0$,
	then
	\begin{align}
		\frac{1}{2} p + \frac{1}{2} (n + 2 - 2j) q + \lambda (j - 1)(n + 1 - j)
		&\geq \frac{p}{2} \left(1 - \frac{|n + 2 - 2 j|}{n} \right) \\
		&\geq 0;
	\end{align}
	in the first inequality,
	we use the fact that $p \geq n |q|$.
	Moreover,
	the equality holds if and only if $\lambda = 0$ and $p = q = 0$.
	Therefore,
	the second variation is always non-positive,
	and equal to zero if and only if $\varphi$ is contained in
	$\ker D_{\eta} + \ker \overline{D_{\eta}}$,
	or equivalently, 
	in $\Re \ker D_{\eta}$,
	since $\varphi$ is real-valued.
\end{proof}

Consider the case $n \geq 2$ and $\lambda < 0$.
Let $\Sigma$ be a closed Riemann surface of genus 2,
and $g_{\Sigma}$ a hyperbolic metric on $\Sigma$.
Then there exists a complex structure on $\Sigma$ such that $g_{\Sigma}$ is K\"{a}hler.
Consider the $n$-dimensional complex manifold $Y = \Sigma^{n}$.
The product metric $g_{Y}$ on $Y$ satisfies $\Ric_{g_{Y}} = - g_{Y}$.
Moreover,
its anti-canonical line bundle $K_{Y}^{-1}$
has the Hermitian metric $h$ induced from $g_{Y}$,
and $\omega_{h} = d d^{c} \log h$ coincides with the K\"{a}hler form of $g_{Y}$.
Hence the tube $S = \{ v \in K_{Y}^{-1} \mid h(v, v) = 1\}$
is a Sasakian $\eta$-Einstein manifold with Einstein constant $- 1$ with respect to $\eta = d^{c} \log h$.
Denote by $\pi_{1}$ the composition of the projection $S \to Y$
and the projection from $Y$ to the first factor $\Sigma$.
Let $f \in C^{\infty}(\Sigma)$ be an eigenfunction of $\Delta_{g_{\Sigma}}$ with eigenvalue $p$.
Then $\pi_{1}^{*} f$ is an element of $\mathscr{H}_{p,0}$,
and
\begin{equation}
	L_{n + 2 - 2 j} (\pi_{1}^{*} f) = \left(\frac{1}{2} p - \frac{(j - 1)(n + 1 - j)}{n + 1} \right) \pi_{1}^{*} f.
\end{equation}

\begin{proposition} \label{prop:spectrum}
	For any $0 < p < 2$,
	there exists a hyperbolic metric $g_{\Sigma}$ on $\Sigma$ such that
	the first positive eigenvalue $\lambda_{1}(\Delta_{g_{\Sigma}})$
	of $\Delta_{g_{\Sigma}}$ is equal to $p$.
\end{proposition}

\begin{proof}
	Consider the space $\mathscr{M}_{-1}$ of hyperbolic metrics on $\Sigma$ with $C^{\infty}$ topology.
	This space is contractible, and in particular connected~\cite[Section 3.4]{Tromba92}.
	Moreover, the map $g_{\Sigma} \mapsto \lambda_{1}(\Delta_{g_{\Sigma}})$
	defines a continuous function on $\mathscr{M}_{-1}$.
	Hence it is sufficient to show that
	\begin{equation}
		\inf_{g_{\Sigma} \in \mathscr{M}_{-1}} \lambda_{1}(\Delta_{g_{\Sigma}}) = 0,
		\qquad \sup_{g_{\Sigma} \in \mathscr{M}_{-1}} \lambda_{1} (\Delta_{g_{\Sigma}}) \geq 2
	\end{equation}
	from the intermediate value theorem.
	The first equality was proved by Buser~\cite[Satz 1]{Buser77}.
	On the other hand,
	it is known that there exists a hyperbolic metric on $\Sigma$
	such that its first eigenvalue is greater than $3.83$~\cite[Section 1.2]{Jenni84};
	see~\cite[Section 5.3]{Strohmaier-Uski13} for a more precise estimate of its value.
	This finishes the proof.
\end{proof}

\begin{proof}[Proof of \cref{thm:second-variation-for-negative-SeE}]
	Since $0 < 2(n - 1)/(n + 1) < 2$,
	we can take a hyperbolic metric $g_{\Sigma}$ on $\Sigma$
	such that $\lambda_{1} (\Delta_{g_{\Sigma}}) = 2(n - 1)/(n + 1)$,
	and a real-valued eigenfunction $0 \neq f$ on $\Sigma$ with eigenvalue $2(n - 1)/(n + 1)$.
	Then $\pi_{1}^{*} f$ is not contained in $\Re \ker D_{\eta}$,
	but $P_{1,1} (\pi_{1}^{*} f) = 0$ since $L_{n - 2} (\pi_{1}^{*} f) = 0$.
	Therefore
	if we take a smooth deformation of $S$ such that $\varphi = f$,
	the second variation of the total $Q$-prime curvature along this deformation is equal to zero
	though this deformation is infinitesimally non-trivial
	as a deformation of CR manifolds.

	Assume that $n = 2m$ for an integer $m \geq 1$.
	Let $0 \neq f \in C^{\infty}(\Sigma)$ be an eigenfunction of $\Delta_{g_{\Sigma}}$ with eigenvalue $p$.
	Then
	\begin{align}
		P_{1, 1}(\pi_{1}^{*} f)
		&= \prod_{j=0}^{2 m + 2} \left(\frac{1}{2} p - \frac{(j - 1)(2m + 1 - j)}{2 m + 1} \right) \pi_{1}^{*} f \\
		&= \left(\frac{1}{2} p - \frac{m^{2}}{2 m + 1} \right)
		\prod_{j=0}^{m} \left(\frac{1}{2} p - \frac{(j - 1)(2 m + 1 - j)}{2 m + 1} \right)^{2} \pi_{1}^{*} f.
	\end{align}
	Hence if we choose $g_{\Sigma}$ and $f$
	such that $p$ is sufficiently small,
	$\pi_{1}^{*} f$ is an eigenfunction of $P_{1, 1}$ with negative eigenvalue.
	Thus there exists a smooth deformation of $S$
	with positive second variation.
\end{proof}

\section{Computation of total \texorpdfstring{$Q$}{Q}-prime curvature}
\label{section:computation-of-total-Q-prime-curvature}

In this section,
we compute the total $Q$-prime curvature for some examples.
To this end,
we first compute the $Q$-prime curvature
for Sasakian $\eta$-Einstein manifolds.
The unbold $Q'_{\eta}$ is defined by
\begin{equation} \label{eq:unbold-Q-prime-curvature}
	Q'_{\eta}
	= \bm{M}_{n + 1, n + 1} \bm{Q}'_{\eta}
	= \bm{Q}'_{\eta} \cdot \bm{h}_{S}^{(n + 1) / (n + 2)} \in C^{\infty}(S).
\end{equation}

\begin{proof}[Proof of \cref{thm:formula-for-Q-prime-curvature}]
	It suffices to compute $\bm{\Delta}^{n + 1} (\log |z^{0}|^{2})$
	for a branched fiber coordinate $z^{0}$ with respect to a flat local coordinate.
	It can be seen that
	\begin{equation}
		\bm{\Delta} (\log |z^{0}|^{2})^{2}
		= - \langle d \log |z^{0}|^{2}, d \log |z^{0}|^{2} \rangle_{\bm{g}}
		= 2 \lambda |z^{0}|^{-2}
	\end{equation}
	and
	\begin{equation}
		\bm{\Delta} |z^{0}|^{-2l}
		= - \langle d (z^{0})^{-l}, d (\bar{z}^{0})^{-l} \rangle_{\bm{g}}
		= l^{2} \lambda |z^{0}|^{-2(l+1)}.
	\end{equation}
	Hence $Q'_{\eta} = 2 (n!)^{2} \lambda^{n + 1}$.
\end{proof}

Let $(S, T^{1,0}S, \eta)$ be a closed Sasakian $\eta$-Einstein manifold
with Einstein constant $(n + 1) \lambda$.
Then the total $Q$-prime curvature $\overline{Q}'(S)$ has the formula
\begin{equation}
	\overline{Q}'(S)
	= 2 (n!)^{2} \lambda^{n + 1} \int_{S} \eta \wedge (d \eta)^{n}
	= 2^{n + 1} (n!)^{3} \lambda^{n + 1} \Vol(S, g_{\eta}),
\end{equation}
where $\Vol(S, g_{\eta})$ is the volume of the Riemannian manifold $(S, g_{\eta})$.
We apply this formula to some examples of Sasakian $\eta$-Einstein manifolds.

\subsection{Sasaki-Einstein manifolds}
\label{subsection:Sasaki-Einstein-manifolds}

If $(S,T^{1,0}S, \eta)$ is a Sasaki-Einstein manifold,
the total $Q$-prime curvature $\overline{Q}'(S)$ is equal to
\begin{equation}
	\overline{Q}'(S)
	= 2^{n + 1} (n!)^{3} \Vol(S, g_{\eta}).
\end{equation}
Hence it is enough to compute the volume of $(S, g_{\eta})$.

\begin{example}[$S^{2 n + 1}$]
	Consider the sphere $(S^{2 n + 1}, T^{1,0} S^{2 n + 1}, \eta_{0})$
	as in \cref{ex:sphere}.
	Then the total $Q$-prime curvature is equal to
	\begin{equation}
		\overline{Q}'(S^{2 n + 1}) = 2^{n + 1} (n !)^{3} \Vol(S^{2 n + 1}, g_{\eta_{0}})
		= 2^{n + 2} (n!)^{2} \pi^{n + 1}.
	\end{equation}
	Here we use the fact that the metric $g_{\eta_{0}}$ is equal to 
	that induced from the Euclidean metric on $\mathbb{C}^{n + 1}$.
\end{example}

\begin{example}[$Y^{p,q}$]
	In the study of the AdS/CFT correspondence,
	Gauntlett-Martelli-Sparks-Waldram
	constructed $5$-dimensional Sasaki-Einstein manifolds $Y^{p, q}$
	for coprime positive integers $q<p$,
	which are diffeomorphic to $S^{2} \times S^{3}$~\cite{Gauntlett-Martelli-Sparks-Waldram04}.
	The total $Q$-prime curvature of $Y^{p, q}$ is given by
	\begin{equation}
		\overline{Q}'(Y^{p, q}) = 2^{3} (2 !)^{3} \Vol(Y^{p, q})
		= \frac{2^{6} q^{2}(2p + (4p^{2} - 3q^{2})^{1/2})}%
		{3p^{2}(3q^{2} - 2p^{2} + p(4p^{2} - 3q^{2})^{1/2})} \pi^{3}.
	\end{equation}
\end{example}

\subsection{Links of affine cones over projective varieties}
\label{subsection:Links-of-affine-cones-over-projective-varieties}

Let $Y$ be an $n$-dimensional smooth projective variety in $\mathbb{CP}^{N}$
and $\Aff(Y) \subset \mathbb{C}^{N+1}$ the affine cone of $Y$,
which may have an isolated singularity at the origin.
Assume that the restriction of the Fubini-Study metric on $Y$ defines a K\"{a}hler-Einstein metric
with Einstein constant $(n + 1) \lambda$.
Then the intersection $S$ of $\Aff(Y)$ and the unit sphere centered at the origin
is a Sasakian $\eta$-Einstein manifold with Einstein constant $(n + 1) \lambda$
with respect to $\eta = \eta_{0}|_{S}$.
In this case,
$S$ is a principal $S^{1}$-bundle over $Y$
and $\eta$ is a connection $1$-form.
Hence the total $Q$-prime curvature $\overline{Q}'(S)$ of $S$ is equal to
\begin{equation}
	\overline{Q}'(S)
	= 2 (n!)^{2} \lambda^{n + 1} \int_{S} \eta \wedge (d \eta)^{n}
	= 4 \pi (n!)^{2} \lambda^{n + 1} \int_{Y} \omega_{FS}^{n}.
\end{equation}
Since the Fubini-Study form $\omega_{FS}$ on $\mathbb{CP}^{N}$
is a representative of the cohomology class $2 \pi c_{1}(\mathcal{O}(1))$,
we obtain
\begin{equation}
	\overline{Q}'(S) = 2^{n + 2} (n!)^{2} \pi^{n + 1} \lambda^{n + 1} \cdot \deg Y,
\end{equation}
where
\begin{equation}
	\deg Y = \int_{Y} c_{1}(\mathcal{O}(1)|_{Y})^{n}.
\end{equation}
Hence it is enough to compute the constant $\lambda$ and the degree $\deg Y$.
Note that $\lambda$ is determined by the formula
$c_{1}(Y) = (n + 1) \lambda \, c_{1}(\mathcal{O}(1)|_{Y})$.

\begin{example}[Fermat hypersurface of degree $2$]
	Let $F \in \mathbb{C}[z^{1}, \dots z^{n + 2}]$ be the homogeneous polynomial defined by
	\begin{equation}
		F(z^{1}, \dots, z^{n + 2}) = (z^{1})^{2} + \dots + (z^{n + 2})^{2}.
	\end{equation}
	The hypersurface $Y \subset \mathbb{CP}^{n + 1}$ defined by $F$
	is called the \emph{Fermat hypersurface of degree $2$}.
	It is known that $\omega_{FS}|_{Y}$ defines
	a K\"{a}hler-Einstein metric on $Y$~\cite[Section 2]{Smyth67}.
	Moreover,
	$\deg Y = 2$
	and $c_{1}(Y) = n c_{1}(\mathcal{O}(1)|_{Y})$.
	Therefore,
	the intersection $S$ of its affine cone and the unit sphere
	has a Sasakian $\eta$-Einstein structure,
	and the total $Q$-prime curvature is given by
	\begin{equation}
		\overline{Q}'(S) = 2^{n + 3} (n!)^{2}  n^{n + 1} (n + 1)^{- (n + 1)} \pi^{n + 1}.
	\end{equation}
\end{example}

\begin{example}[Grassmannian manifold]
	Let $G(k, n)$ be the complex Grassmannian manifold;
	that is, the space of $k$-dimensional $\mathbb{C}$-linear subspaces in $\mathbb{C}^{n}$.
	By the Pl\"{u}cker embedding,
	identify $G(k, n)$ with its image in $\mathbb{CP}^{\binom{n}{k}-1}$.
	Note that its affine cone $\Aff(G(k, n))$ is the affine variety defined by the Pl\"{u}cker relations.
	A direct calculation shows that
	the Fubini-Study metric on $\mathbb{CP}^{\binom{n}{k}-1}$
	induces a K\"{a}hler-Einstein metric on $G(k, n)$ with Einstein constant $n$.
	Hence the intersection $S = \Aff(G(k, n)) \cap S^{2\binom{n}{k} - 1}$
	is a Sasakian $\eta$-Einstein manifold
	with respect to $\eta_{0}|_{S}$.
	The degree of $G(k, n)$ can be computed from so-called Schubert calculus.
	The result is 
	\begin{equation}
		\deg G(k, n) = (k(n-k))! \prod_{i=1}^{k} \frac{(i-1)!}{(n-k+i-1)!};
	\end{equation}
	see for example~\cite[Example 14.7.11]{Fulton98}. 
	Thus, the total $Q$-prime curvature is given by
	\begin{equation}
		\overline{Q}'(S)
		= 2^{k(n-k)+2} ((k(n-k))!)^{2}
		\left(\frac{n \pi}{k(n-k)+1}\right)^{k(n-k)+1}
		\cdot \deg G(k, n).
	\end{equation}
\end{example}

\medskip

\section*{Acknowledgements}
The author is grateful to his supervisor Kengo Hirachi
for various helpful suggestions.
He also thanks Paul Yang and Jeffrey Case for invaluable comments.
A part of this work was carried out during his visit to Princeton University
with the support from The University of Tokyo/Princeton University
Strategic Partnership Teaching and Research Collaboration Grant,
and from the Program for Leading Graduate Schools, MEXT, Japan.
This work was also supported by JSPS Research Fellowship for Young Scientists
and KAKENHI Grant Number 16J04653.



\begin{bibdiv}
\begin{biblist}

\bib{Akahori-Garfield-Lee02}{article}{
      author={Akahori, Takao},
      author={Garfield, Peter~M.},
      author={Lee, John~M.},
       title={Deformation theory of 5-dimensional {CR} structures and the
  {R}umin complex},
        date={2002},
        ISSN={0026-2285},
     journal={Michigan Math. J.},
      volume={50},
      number={3},
       pages={517\ndash 549},
         url={http://dx.doi.org/10.1307/mmj/1039029981},
      review={\MR{1935151}},
}

\bib{Besse87}{book}{
      author={Besse, Arthur~L.},
       title={Einstein manifolds},
      series={Ergebnisse der Mathematik und ihrer Grenzgebiete (3) [Results in
  Mathematics and Related Areas (3)]},
   publisher={Springer-Verlag, Berlin},
        date={1987},
      volume={10},
        ISBN={3-540-15279-2},
         url={https://doi.org/10.1007/978-3-540-74311-8},
      review={\MR{867684}},
}

\bib{Boyer-Galicki08}{book}{
      author={Boyer, Charles~P.},
      author={Galicki, Krzysztof},
       title={Sasakian geometry},
      series={Oxford Mathematical Monographs},
   publisher={Oxford University Press, Oxford},
        date={2008},
        ISBN={978-0-19-856495-9},
      review={\MR{2382957}},
}

\bib{Branson-Fontana-Morpurgo13}{article}{
      author={Branson, Thomas~P.},
      author={Fontana, Luigi},
      author={Morpurgo, Carlo},
       title={Moser-{T}rudinger and {B}eckner-{O}nofri's inequalities on the
  {CR} sphere},
        date={2013},
        ISSN={0003-486X},
     journal={Ann. of Math. (2)},
      volume={177},
      number={1},
       pages={1\ndash 52},
         url={http://dx.doi.org/10.4007/annals.2013.177.1.1},
      review={\MR{2999037}},
}

\bib{Buchweitz-Millson97}{article}{
      author={Buchweitz, Ragnar-Olaf},
      author={Millson, John~J.},
       title={C{R}-geometry and deformations of isolated singularities},
        date={1997},
        ISSN={0065-9266},
     journal={Mem. Amer. Math. Soc.},
      volume={125},
      number={597},
       pages={viii+96},
         url={https://doi.org/10.1090/memo/0597},
      review={\MR{1355712}},
}

\bib{Buser77}{article}{
      author={Buser, Peter},
       title={Riemannsche {F}l\"achen mit {E}igenwerten in {$(0,$} {$1/4)$}},
        date={1977},
        ISSN={0010-2571},
     journal={Comment. Math. Helv.},
      volume={52},
      number={1},
       pages={25\ndash 34},
         url={http://dx.doi.org/10.1007/BF02567355},
      review={\MR{0434961}},
}

\bib{Cap-Slovak-Soucek01}{article}{
      author={\v{C}ap, Andreas},
      author={Slov\'ak, Jan},
      author={Sou\v{c}ek, Vladim\'{\i}r},
       title={Bernstein-{G}elfand-{G}elfand sequences},
        date={2001},
        ISSN={0003-486X},
     journal={Ann. of Math. (2)},
      volume={154},
      number={1},
       pages={97\ndash 113},
         url={https://doi.org/10.2307/3062111},
      review={\MR{1847589}},
}

\bib{Case-Gover17}{unpublished}{
      author={Case, Jeffrey~S.},
      author={Gover, A.~Rod},
       title={The {$P'$}-operator, the {$Q'$}-curvature, and the {CR} tractor
  calculus},
        note={{\tt arXiv:1709.08057}},
}

\bib{Case-Yang13}{article}{
      author={Case, Jeffrey~S.},
      author={Yang, Paul},
       title={A {P}aneitz-type operator for {CR} pluriharmonic functions},
        date={2013},
        ISSN={2304-7909},
     journal={Bull. Inst. Math. Acad. Sin. (N.S.)},
      volume={8},
      number={3},
       pages={285\ndash 322},
      review={\MR{3135070}},
}

\bib{Cheng-Lee90}{article}{
      author={Ch\^eng, Jih~Hsin},
      author={Lee, John~M.},
       title={The {B}urns-{E}pstein invariant and deformation of {CR}
  structures},
        date={1990},
        ISSN={0012-7094},
     journal={Duke Math. J.},
      volume={60},
      number={1},
       pages={221\ndash 254},
         url={http://dx.doi.org/10.1215/S0012-7094-90-06008-9},
      review={\MR{1047122}},
}

\bib{Fefferman-Graham12}{book}{
      author={Fefferman, Charles},
      author={Graham, C.~Robin},
       title={The ambient metric},
      series={Annals of Mathematics Studies},
   publisher={Princeton University Press, Princeton, NJ},
        date={2012},
      volume={178},
        ISBN={978-0-691-15313-1},
      review={\MR{2858236}},
}

\bib{Fulton98}{book}{
      author={Fulton, William},
       title={Intersection theory},
     edition={Second},
      series={Ergebnisse der Mathematik und ihrer Grenzgebiete. 3. Folge. A
  Series of Modern Surveys in Mathematics [Results in Mathematics and Related
  Areas. 3rd Series. A Series of Modern Surveys in Mathematics]},
   publisher={Springer-Verlag, Berlin},
        date={1998},
      volume={2},
        ISBN={3-540-62046-X; 0-387-98549-2},
         url={http://dx.doi.org/10.1007/978-1-4612-1700-8},
      review={\MR{1644323}},
}

\bib{Gauntlett-Martelli-Sparks-Waldram04}{article}{
      author={Gauntlett, Jerome~P.},
      author={Martelli, Dario},
      author={Sparks, James},
      author={Waldram, Daniel},
       title={Sasaki-{E}instein metrics on {$S\sp 2\times S\sp 3$}},
        date={2004},
        ISSN={1095-0761},
     journal={Adv. Theor. Math. Phys.},
      volume={8},
      number={4},
       pages={711\ndash 734},
         url={http://projecteuclid.org/euclid.atmp/1117750699},
      review={\MR{2141499}},
}

\bib{Gover06}{article}{
      author={Gover, A.~Rod},
       title={Laplacian operators and {$Q$}-curvature on conformally {E}instein
  manifolds},
        date={2006},
        ISSN={0025-5831},
     journal={Math. Ann.},
      volume={336},
      number={2},
       pages={311\ndash 334},
         url={http://dx.doi.org/10.1007/s00208-006-0004-z},
      review={\MR{2244375}},
}

\bib{Gover-Graham05}{article}{
      author={Gover, A.~Rod},
      author={Graham, C.~Robin},
       title={C{R} invariant powers of the sub-{L}aplacian},
        date={2005},
        ISSN={0075-4102},
     journal={J. Reine Angew. Math.},
      volume={583},
       pages={1\ndash 27},
         url={http://dx.doi.org/10.1515/crll.2005.2005.583.1},
      review={\MR{2146851}},
}

\bib{Graham84}{article}{
      author={Graham, C.~Robin},
       title={Compatibility operators for degenerate elliptic equations on the
  ball and {H}eisenberg group},
        date={1984},
        ISSN={0025-5874},
     journal={Math. Z.},
      volume={187},
      number={3},
       pages={289\ndash 304},
         url={http://dx.doi.org/10.1007/BF01161947},
      review={\MR{757471}},
}

\bib{Graham-Hirachi05}{incollection}{
      author={Graham, C.~Robin},
      author={Hirachi, Kengo},
       title={The ambient obstruction tensor and {$Q$}-curvature},
        date={2005},
   booktitle={Ad{S}/{CFT} correspondence: {E}instein metrics and their
  conformal boundaries},
      series={IRMA Lect. Math. Theor. Phys.},
      volume={8},
   publisher={Eur. Math. Soc., Z\"urich},
       pages={59\ndash 71},
         url={http://dx.doi.org/10.4171/013-1/3},
      review={\MR{2160867}},
}

\bib{Graham-Lee88}{article}{
      author={Graham, C.~Robin},
      author={Lee, John~M.},
       title={Smooth solutions of degenerate {L}aplacians on strictly
  pseudoconvex domains},
        date={1988},
        ISSN={0012-7094},
     journal={Duke Math. J.},
      volume={57},
      number={3},
       pages={697\ndash 720},
         url={http://dx.doi.org/10.1215/S0012-7094-88-05731-6},
      review={\MR{975118}},
}

\bib{Guillarmou-Moroianu-Schlenker}{article}{
      author={Guillarmou, Colin},
      author={Moroianu, Sergiu},
      author={Schlenker, Jean-Marc},
       title={The renormalized volume and uniformization of conformal
  structures},
        date={2016},
     journal={J. Inst. Math. Jussieu},
       pages={1\ndash 60},
         url={https://doi.org/10.1017/S1474748016000244},
}

\bib{Hirachi14}{article}{
      author={Hirachi, Kengo},
       title={{$Q$}-prime curvature on {CR} manifolds},
        date={2014},
        ISSN={0926-2245},
     journal={Differential Geom. Appl.},
      volume={33},
      number={suppl.},
       pages={213\ndash 245},
         url={http://dx.doi.org/10.1016/j.difgeo.2013.10.013},
      review={\MR{3159959}},
}

\bib{Hirachi-Marugame-Matsumoto17}{article}{
      author={Hirachi, Kengo},
      author={Marugame, Taiji},
      author={Matsumoto, Yoshihiko},
       title={Variation of total {Q}-prime curvature on {CR} manifolds},
        date={2017},
        ISSN={0001-8708},
     journal={Adv. Math.},
      volume={306},
       pages={1333\ndash 1376},
         url={http://dx.doi.org/10.1016/j.aim.2016.11.005},
      review={\MR{3581332}},
}

\bib{Jenni84}{article}{
      author={Jenni, Felix},
       title={\"{U}ber den ersten {E}igenwert des {L}aplace-{O}perators auf
  ausgew\"ahlten {B}eispielen kompakter {R}iemannscher {F}l\"achen},
        date={1984},
        ISSN={0010-2571},
     journal={Comment. Math. Helv.},
      volume={59},
      number={2},
       pages={193\ndash 203},
         url={http://dx.doi.org/10.1007/BF02566345},
      review={\MR{749104}},
}

\bib{Lee88}{article}{
      author={Lee, John~M.},
       title={Pseudo-{E}instein structures on {CR} manifolds},
        date={1988},
        ISSN={0002-9327},
     journal={Amer. J. Math.},
      volume={110},
      number={1},
       pages={157\ndash 178},
         url={http://dx.doi.org/10.2307/2374543},
      review={\MR{926742}},
}

\bib{Marugame18}{article}{
      author={Marugame, Taiji},
       title={Some remarks on the total {CR} {$Q$} and {$Q'$}-curvatures},
        date={2018},
        ISSN={1815-0659},
     journal={SIGMA Symmetry Integrability Geom. Methods Appl.},
      volume={14},
       pages={Paper No. 010, 8},
         url={https://doi.org/10.3842/SIGMA.2018.010},
      review={\MR{3763373}},
}

\bib{Matsumoto13}{article}{
      author={Matsumoto, Yoshihiko},
       title={A {GJMS} construction for 2-tensors and the second variation of
  the total {$Q$}-curvature},
        date={2013},
        ISSN={0030-8730},
     journal={Pacific J. Math.},
      volume={262},
      number={2},
       pages={437\ndash 455},
         url={http://dx.doi.org/10.2140/pjm.2013.262.437},
      review={\MR{3069069}},
}

\bib{Smyth67}{article}{
      author={Smyth, Brian},
       title={Differential geometry of complex hypersurfaces},
        date={1967},
        ISSN={0003-486X},
     journal={Ann. of Math. (2)},
      volume={85},
       pages={246\ndash 266},
         url={http://dx.doi.org/10.2307/1970441},
      review={\MR{0206881}},
}

\bib{Sparks11}{incollection}{
      author={Sparks, James},
       title={Sasaki-{E}instein manifolds},
        date={2011},
   booktitle={Surveys in differential geometry. {V}olume {XVI}. {G}eometry of
  special holonomy and related topics},
      series={Surv. Differ. Geom.},
      volume={16},
   publisher={Int. Press, Somerville, MA},
       pages={265\ndash 324},
         url={http://dx.doi.org/10.4310/SDG.2011.v16.n1.a6},
      review={\MR{2893680}},
}

\bib{Strohmaier-Uski13}{article}{
      author={Strohmaier, Alexander},
      author={Uski, Ville},
       title={An algorithm for the computation of eigenvalues, spectral zeta
  functions and zeta-determinants on hyperbolic surfaces},
        date={2013},
        ISSN={0010-3616},
     journal={Comm. Math. Phys.},
      volume={317},
      number={3},
       pages={827\ndash 869},
         url={http://dx.doi.org/10.1007/s00220-012-1557-1},
      review={\MR{3009726}},
}

\bib{Tromba92}{book}{
      author={Tromba, Anthony~J.},
       title={Teichm\"uller theory in {R}iemannian geometry},
      series={Lectures in Mathematics ETH Z\"urich},
   publisher={Birkh\"auser Verlag, Basel},
        date={1992},
        ISBN={3-7643-2735-9},
         url={http://dx.doi.org/10.1007/978-3-0348-8613-0},
        note={Lecture notes prepared by Jochen Denzler},
      review={\MR{1164870}},
}



\end{biblist}
\end{bibdiv}
\end{document}